\documentclass[10pt]{article}
\usepackage[english]{babel}
\usepackage{amsmath, amsthm, amsfonts, mathrsfs, amsfonts,mathtools}
\usepackage{bm}
\usepackage{amscd}
\usepackage{amssymb}
\usepackage{enumerate} 
\usepackage{enumitem}
\usepackage{multirow}
\usepackage{graphicx}
\usepackage{xypic}
\usepackage{color}
\usepackage{enumitem}
\usepackage{ulem}

\usepackage[a4paper]{geometry}
\usepackage{titlesec}
\usepackage{hyperref}
\usepackage{fancyhdr}
\usepackage{url}
\usepackage{float}

\titleformat{\section}
  {\normalfont\Large\bfseries}{\thesection}{1em}{}[{\titlerule[0.8pt]}]

\newcommand{\Z}{\ensuremath{\mathbb{Z}}}

\newcommand{\A}{\ensuremath{\mathbb{A}}}

\newcommand{\cor}[1]{\mathcal{#1}}

\newcommand{\graffe}[1]{\{#1\}}

\newcommand{\gen}[1]{\langle #1\rangle}

\DeclareMathOperator{\ZG}{Z} %centre of a group
\DeclareMathOperator{\Cyc}{C} %centralizer
\DeclareMathOperator{\wt}{wt} %width
\DeclareMathOperator{\ab}{ab}
\DeclareMathOperator{\SG}{SG}
\DeclareMathOperator{\Jen}{L}

\DeclarePairedDelimiter{\ceil}{\lceil}{\rceil}

\newcommand{\ls}[2]{\gamma_{#1}(#2)}

\newtheorem{definition}{Definition}[section]
\newtheorem{lemma}[definition]{Lemma}
\newtheorem{theorem}[definition]{Theorem}
\newtheorem{proposition}[definition]{Proposition}
\newtheorem{corollary}[definition]{Corollary}

\newtheorem*{thx}{Acknowledgements}

\newenvironment{remark}[1][]{\refstepcounter{definition}\par\medskip
   \noindent \textbf{Remark~\thedefinition. #1} \rmfamily}{\medskip}

\newenvironment{example}[1][]{\refstepcounter{definition}\par\medskip
   \noindent \textbf{Example~\thedefinition. #1} \rmfamily}{\medskip}

\newenvironment{qs}[1][]{\refstepcounter{definition}\par\medskip
   \noindent \textbf{Question~\thedefinition. #1} \rmfamily}{\medskip}
   
   \newenvironment{qs*}[1][]{\par\medskip
   \noindent \textbf{#1} \rmfamily}{\medskip}

\numberwithin{equation}{section}

\author{Leo Margolis and Mima Stanojkovski\footnote{This research has been partially supported by the FWO (Research Foundation Flanders) and CIRM-FBK Trento in connection to their Research in Pairs Program.\newline
}}

\title{
On the Modular Isomorphism Problem for groups of class $3$ and obelisks 
}

\begin{document}

\maketitle

\begin{abstract}
We study the Modular Isomorphism Problem applying a combination of existing and new techniques. We make use of the small group algebra to give a positive answer for two classes of groups of nilpotency class $3$. We also introduce a new approach to derive properties of the lower central series of a finite $p$-group from the structure of the associated modular group algebra. Finally, we study the class of so-called $p$-obelisks which are highlighted by recent computer-aided investigations of the problem.
\end{abstract}

%\tableofcontents

\normalem % This makes the emphazise normal. I don't remember why I included ulem, but that is the reason...

\section{Introduction}
The question of how strongly the group ring $RG$ of a group $G$ over a commutative ring $R$ dictates the structure of $G$ is still far from being well-understood. The most classical version of this problem that is still open today, in its full generality, states:

\begin{qs*}[Modular Isomorphism Problem (MIP).]
Let $k$ be a field of characteristic $p>0$ and let $G$ be a finite $p$-group. Let $H$ be another group. Does a ring isomorphism between $kG$ and $kH$ imply a group isomorphism between $G$ and $H$? In symbols:
\[
kG\cong kH \ \Longrightarrow \ G\cong H?
\]
\end{qs*}
\vspace{-13pt}\\
\noindent
If $H$ is a subgroup of the unit group of $kG$ such that $kG = kH$, then $H$ is said to be a \emph{group base} of $kG$. The (MIP) can thus be reformulated as: are all group bases of a modular group algebra of a finite $p$-group isomorphic?

An explicit formulation of the (MIP) can already be found in a famous survey of Brauer~\cite{Bra63}. Since then, the problem has been studied by many researchers: we refer to \cite{HS06, EK11} for an overview of most known results. New points of view and approaches to the problem can be found in the very recent papers \cite{BK19, Sak20, BdR20, MM20}. We study the (MIP) for certain classes of groups, inspired by some of these recent results and accessing the problem from different directions. 

We start by recalling the construction of the so-called small group algebra of $kG$ and its employment in the study of the (MIP):  among its biggest merits are Sandling's  positive solution of the (MIP) for groups with central elementary abelian derived subgroup \cite{San89} and Salim and Sandling's solution for groups of order $p^5$ \cite{SS96p5}.

%\ChMima{I would cut this part, as we discuss it in detail later:
% Denote by $(\gamma_i(G)_{i\geq 1})$ the lower central series of a group $G$. The small group algebra is also used in \cite{MM20} to handle the $2$-generated groups satisfying $\gamma_2(G)^p\gamma_4(G) = 1$. More generally, the small group algebra is especially powerful for groups $G$ satisfying $\gamma_2(G)^p\gamma_4(G) = 1$ and 
%indeed for all groups of order $p^5$ satisfying this condition this was the way the (MIP) was proven \cite{Sal93, SS96p5} (for odd $p$).} \ChLeo{I agree}

We use the small group algebra to positively solve the (MIP) for two new classes of groups, cf.\ Theorem~\ref{th:MaxAbelianCentDerSubg} and Theorem~\ref{th:K_G}.
 In particular, for $p$ odd, our work completely covers 5 of 43 isoclinism classes of groups of order $p^6$, while 10 other classes are covered by known results.  
We also show that the strategy of Salim and Sandling for groups of order $p^5$ is not directly translatable to groups of order $p^6$ by
exhibiting a pair of non-isomorphic groups, of class $3$ and with elementary abelian derived subgroups,  whose small group algebras have isomorphic groups of units, cf. Example~\ref{ex:553vs554}.  We believe that the methods used to prove Theorems~\ref{th:MaxAbelianCentDerSubg} and \ref{th:K_G} can be applied to other families of groups with maximal subgroups of relatively small class.

We continue our investigation with the determination of a new invariant for certain $2$-generated groups of class $3$. %\ChMima{(Can we say that our invariant should give us $|\gamma_3(G):\gamma_3(G)^p\gamma_4(G)|$ or something like this? I don't want to assume that $\gamma_3$ is elementary abelian basically)}
In relation to this, we stress that, though the lower central series is a fundamental object in the study of finite $p$-groups, it is still quite unclear which properties of it are detected by the group algebra. For instance, it is not known in general whether group bases for the same algebra have the same nilpotency class and, even if they did, it is not known whether corresponding elements of the lower central series are isomorphic. 
Moreover, Example 2.1 from \cite{BK19}  makes it clear that the ideals corresponding to the lower central series in $kG$ are not canonical and thus they are probably not the best candidates for recognizing the lower central series of the group.
Via new quotient-in-quotient embedding techniques, 
we show that, for certain $2$-generated groups of nilpotency class $3$, % with $\gamma_3(G)$ \ChMima{(problem:lower not defined -- shall we or ok?)} \ChLeo{I would say ``the last non-trivial member of the lower central series''} of exponent $p$, then 
the isomorphism types of the members of the lower central series of any group base are determined by $kG$, cf. Theorem~\ref{th:2gen-class3-orderG3}. 
With the aid of computer algebra systems, it can be checked that the last theorem 
does not follow from known results on the (MIP), contributing to the mystery around the connection between the lower central series of the group and the Lie structure of the  modular group algebra.
  
We conclude the present paper with the study of the (MIP) for a family of groups called $p$-obelisks. These groups are highlighted in computer-aided investigations for groups of order $5^6$ \cite{MM20} and, in some sense, are among  the   groups of order $5^6$ that are ``closest'' to being counterexamples to the (MIP). We show that if $kG \cong kH$ and $G$ is a $p$-obelisk, then so is $H$, cf. Theorem~\ref{theo_IsOb}.

%
%\subsection{Motivation}
%
%
%\begin{enumerate}[label=$(\arabic*)$]
%\item $2$-generated class $2$
%\item $2$-generated class $3$ and $\ls{2}{G}^p=1$
%\item Salim-Sandling $p^5$ \cite{SS96MC}, \cite{SS95}, \cite{SS96p5}: methods, generalization and limits 
%\item $p^6$
%\item Computational results and $p$-obelisks
%\end{enumerate}

\begin{thx}
The research in this paper has been conducted at various institutions with different sponsors to whom we wish to express our gratitude. We thank Eric Jespers from the Vrije Universiteit Brussel, Christopher Voll from the Universit\"at Bielefeld, and Bernd Sturmfels from the Max-Planck-Institute for Mathematics in the Sciences.
We, moreover, wish to thank CIRM-FBK Trento for the financial support and hospitality during our Research in Pairs visit in February 2019.
\end{thx}

\section{Preliminaries and notation}\label{sec:Prelim}

Throughout the article, the following assumptions will hold. Let $p$ be a prime number. We denote by $G$ a $p$-group and by $k$ the field of $p$ elements. The (modular) group algebra of $G$ over $k$ is denoted by $kG$. Note that this provides the most general situation for the (MIP): indeed, for any field extension $F/k$, an isomorphism $kG \cong kH$ induces, via extension of scalars, an isomorphism $FG \cong FH$.

\subsection{Group theoretical notions}\label{sec:gp not}
We denote the order of an element $g \in G$ by $|g|$ and the order of $G$ by $|G|$. The index of a subgroup $X$ in $G$ is denoted by $|G:X|$.
The lower central series of $G$ is denoted by $(\ls{i}{G})_{i\geq 1}$ and its upper central series by $(\ZG_i(G))_{i\geq 1}$. 
We write $\ZG(G)$ instead instead of $\ZG_1(G)$ for the centre of $G$. We write, moreover, $\Phi(G)$ for the Frattini subgroup of $G$. We define
$$\mu_p(G)=\gen{x \in G \mid x^p=1} \ \textup{ and } \ G^p=\gen{x^p \mid x\in G}.$$   
In this paper, conjugation is considered to be a \emph{right} action and, consequently, commutators are defined and grouped in the following way. For $x,y,z \in G$, we set 
$$[x,y] = x^{-1}y^{-1}xy \ \textup{ and } \ [x,y,z] = [[x,y],z].$$
For the convenience of the reader, we include the well-known commutator formulas; see for example \cite[III, 1.2 Hilfssatz]{Hup67}. 

\begin{lemma}\label{lemma:comm-formulas} 
Let $X$ be a group. For elements $a,b,c \in X$, the following hold:
\begin{align*}
[a,bc]  = [a,c][a,b]^c \ \textup{ and } \
[ab,c] = [a,c]^b[b,c].
\end{align*} 
\end{lemma}

\noindent
Given the special role that the subgroup $\ZG(G)\cap\gamma_2(G)$ will play for us, we associate to it the non-standard notation $\Gamma(G)=\ZG(G)\cap \gamma_2(G)$.

When considering explicit examples, we will refer to the SmallGroups library in GAP \cite{GAP, SmallGroupLibrary} and will denote 
by $\SG(n,m)$ the $m$-th group of order $n$ in the catalog. 

\subsection{Algebra notions}\label{sec:alg not}

Until the end of Section \ref{sec:alg not}, let $\A$ be a $k$-algebra with fixed $k$-basis $\cor{B}$. An element $b\in\cor{B}$ is said to be \emph{in the support of} an element $a\in\A$ if the image of $a$ under the natural projection $\A\rightarrow kb$ is non-zero.

The algebra $\A$ has a natural Lie algebra structure via defining, for all $a,b\in\A$, Lie brackets as $[a,b]=ab-ba$. Given any two subsets $X,Y$ of $\A$, denote by $[X,Y]$ the set of all elements $[x,y]$ with $x\in X$ and $y\in Y$.  
If $\A$ is a local algebra, we denote, moreover, by $V(\A)$ the elements of $\A$ which map to the identity modulo the radical of $\A$. In case $\A$ equals $kG$ or one of its quotients, $V(\A)$ is usually called the \emph{group of normalized units} of $\A$. 

For each element $g \in G$, write now $\bar{g} = g-1 \in kG$ and, for each subset $X$ of $G$, write $\overline{X}$ for the set $\graffe{\bar{x} \mid x\in X}$. Then $kG$ is a local algebra whose radical coincides with the \emph{augmentation ideal}
\[I(kG)= V(kG)-1 = \left\{\sum_{g \in G} \lambda_g g \mid \lambda_g \in k, \ \sum_{g \in G}\lambda_g = 0 \right\},  \]
which is easily seen to be generated, as a $kG$-ideal, by the elements $\bar{g}$ as $g$ varies in $G$.
The augmentation ideal is a nilpotent ideal of $kG$ and hence, for every element $x \in I(kG)$ different from $0$, there exists a minimal number $w$ such that $x \in I(kG)^w\setminus I(kG)^{w+1}$. Such minimal $w$ is called the \emph{weight} of $x$ in $kG$, written 
$$\wt(x)=w.$$
We define, additionally, the \emph{weight} of an element $g \in G$ to be the weight of $\bar{g}$ in $kG$, i.e.\ $\wt(g)=\wt(\bar{g})$. 

The following lemma collects some straightforward equalities holding in $I(kG)$. 

\begin{lemma}\label{lemma:BasicJenningBasisFormulas}
Let $g,h \in G$. Then the following hold:
\begin{align*}
\overline{gh} &= \bar{g} + \bar{h} + \bar{g}\bar{h} , \\
\bar{h}\bar{g} &= \bar{g}\bar{h} + (1 + \bar{g} + \bar{h} + \bar{g}\bar{h})\overline{[h,g]}.
\end{align*}
\end{lemma}

\noindent
We conclude this section by observing that
\[kGI(k\gamma_2(G)) = \gen{x \cdot y \mid x \in kG,\ y \in I(k\gamma_2(G))}_{kG},\]
the ideal of $kG$ that is generated by the elements of $I(k\gamma_2(G))$, 
is minimal among the ideals $J$ of $kG$ such that the quotient $kG/J$ is commutative \cite[Lemma 3.5]{San84}.

\subsection{Jennings theory}\label{sec:Ideals}
Many attempts of tackling the (MIP) rely, in one way or the other, on the theory of Jennings \cite{Jen41} providing a $k$-basis of $I(kG)$  that is well-suited for calculations in a way that we now explain. For more detail, we refer the reader to \cite{DdSMS99, HS06}.

The \emph{dimension subgroups} of $G$ are defined, for each positive integer $n$, by 
\[
D_n(G) = G \cap (1+I(kG)^n)
\]
and filter the elements of (the group base) $G$ according to their weights, as defined in Section \ref{sec:alg not}. 
A first characterization of the dimension subgroups was provided by Jennings himself who showed that the series of dimension subgroups coincides with the Brauer-Jennings-Zassenhaus series, i.e.\ that 
\[
D_n(G)=\begin{cases}
G & \textup{ if } n=1, \\
[G,D_{n-1}(G)]D_{\ceil{\frac{n}{p}}}(G)^p & \textup{ if } n\geq 2.
\end{cases}
\]
Another celebrated characterization of the dimension subgroups of $G$ is given in terms of the members of the lower central series and reads
\[D_n(G) = \prod_{ip^j \geq n} \gamma_i(G)^{p^j} = D_{\ceil{\frac{n}{p}}}(G)^p \gamma_n(G); \]
cf.\  \cite[Theorem 11.2, Proposition 11.12]{DdSMS99}.
From the last characterization, it follows that every quotient $D_n(G)/D_{n+1}(G)$ is elementary abelian and that, for each choice of $m$ and $n$, one has 
\begin{equation}\label{eq:prop-dim}
[D_m(G),D_n(G)]\subseteq D_{m+n}(G) \textup{ and } D_n(G)^p\subseteq D_{np}(G);
\end{equation}
see for instance \cite[Section 11.2]{DdSMS99}.  In particular, the graded $k$-module
\[
\Jen(G)=\bigoplus_{n\geq 1}D_n(G)/D_{n+1}(G)
\]
is naturally a restricted $k$-Lie algebra \cite[Chapter 12]{DdSMS99}.

For each positive integer $n$, it is not difficult to see, via the obvious maps, that $D_n(G)/D_{n+1}(G)$ embeds as a $k$-subspace of $I(kG)^n/I(kG)^{n+1}$.
We derive that, if $g_1,\ldots,g_m \in D_n(G)$ are such that their images in $D_n(G)/D_{n+1}(G)$ form a basis over $k$, then their images $\overline{g_1},\ldots,\overline{g_m}\in I(kG)^{n}$ are linearly independent over $k$ when viewed in $I(kG)^n/I(kG)^{n+1}$. Additionally, the following hold
\[
\wt(\bar{g})=w \Longleftrightarrow \wt(g)=w \Longleftrightarrow g \in D_w(G)\setminus D_{w+1}(G)
\]
and, as a consequence of \eqref{eq:prop-dim}, for each $x,y\in G$, one also has 
\[
\wt([x,y])\geq \wt(x)+\wt(y) \textup{ and } \wt(x^p)\geq p\wt(x).
\]
We write $d_n=\dim_k D_n(G)/D_{n+1}(G)$ and let  $t$ be minimal with the property that $D_{t+1}(G) = 1$. Set $d_0=0$ and $m=d_1+\ldots+ d_{t}=\log_p|G|$. A \emph{Jennings tuple} is a tuple 
$$(g_1, \ldots, g_{d_1} \mid g_{d_1+1}, \ldots, g_{d_1+d_2}\mid \ldots \mid g_{m-d_{t}+1}, \ldots, g_m) $$
of elements of $G$ with the property that for all $n \geq 0$
\[
g_{d_0+\ldots+d_n+1}D_{n+2}(G),\ldots, g_{d_0+\ldots+d_{n+1}}D_{n+2}(G) \textup{ is a basis of } D_{n+1}(G)/D_{n+2}(G).
\]
It follows that any element $g$ of $G$ can be written uniquely as a product 
$$g=\prod_{i=1}^m g_i^{\alpha_i} \textup{ with }
0 \leq \alpha_1,\ldots,\alpha_m \leq p-1.
$$ 
If $(g_1,\ldots,g_m)$ is a Jennings tuple of $G$, then the set 
\[ \left\{ \prod_{i=1}^m \overline{g_i}^{\alpha_i} \mid 0\leq \alpha_1, \ldots, \alpha_m \leq p-1 \right\} \]
forms a $k$-basis of $I(kG)$ which we will refer to as the \emph{Jennings basis on} $(g_1,\ldots,g_m)$. A \emph{Jennings basis} of $I(kG)$ is a $k$-basis of $I(kG)$ that can be defined on a Jennings tuple. Observe that any Jennings basis of $I(kG)$ can be extended to a $k$-basis of $kG$ by adding the element $1$. The advantage of working with a Jennings basis is that calculations can be translated directly from the relations of the group elements. 

Let now $g\in G$. It is possible that multiple powers of $G$ appear in a Jennings tuple of $G$. For instance, if $G$ is cyclic and $g\notin \Phi(G)$, then $(g\mid g^p\mid \ldots \mid g^{\exp(G)/p})$ is a Jennings tuple of $G$. For a Jennings tuple $(g_1,\ldots,g_m)$, %let $y$ be an element of the Jennings basis on this tuple. We
we say that $\overline{g}$ \emph{is a factor of} the element 
$$y = \prod_{i=1}^m \overline{g_i}^{\alpha_i}, \textup{ with } 0 \leq \alpha_1,\ldots,\alpha_m \leq p-1,$$
if there exist non-negative integers $i$ and $j$ such that $g^{p^j} = g_i$ and $\alpha_i \neq 0$. 

Another convenient feature of Jennings bases is that, in some cases, they can be constructed to ``remember'' the group structure. The following fact, which readily follows from the second formula in Lemma~\ref{lemma:BasicJenningBasisFormulas}, will be useful to us. Let $(c_1,...,c_m)$ be a Jennings tuple of $\gamma_2(G)$ such that the elements $c_1,\ldots,c_m$ are also elements of a Jennings tuple ${\bf g}$ of $G$. Then, with respect to the Jennings basis on ${\bf g}$, an element $x \in I(kG)$ lies in the ideal $kGI(k\gamma_2(G))$ if and only if each element in the support of $x$ contains a factor from $\{\overline{c_1},\ldots,\overline{c_m} \}$. 

We conclude the current section by defining a class of ideals of $kG$ that are strictly connected to the dimension subgroups of $G$. 
We define the \emph{Zassenhaus ideals} of $kG$ via their characterization, given by Passi and Segal in \cite{PS72}: for each positive integer $n$, the \emph{$n$-th Zassenhaus ideal of $kG$} is
\[H_n(kG) = \overline{D_n(G)} + I(kG)^{n+1}. \]
The Zassenhaus ideals are canonical ideals of $kG$, i.e.\ each
$H_n(kG)$ can be defined independently of the group base of $kG$, cf.\ \cite[Section 1.3]{HS06}. 
Moreover, thanks to their definition, Zassenhaus ideals allow us to isolate, within quotients of the form $I(kG)^n/I(kG)^{n+1}$, specific contributions of the chosen Jennings basis.  
More precisely, if $U$ is the subspace of $kG$ that is spanned by all elements of a Jennings basis of $I(kG)$ not lying in $\overline{G}=G-1$, then, for each positive integer $n$, one has
\[I(kG)^n/I(kG)^{n+1} = H_n(kG)/I(kG)^{n+1} \oplus (U \cap I(kG)^n + I(kG)^{n+1})/I(kG)^{n+1}. \]
For more on this, see for instance \cite[Section 2.4]{HS06}.

%\ChLeo{Also this seems trivial to me, but it's also in the proof of Theorem 6.17 in Sehgal's pink book}

\subsection{Group theoretical invariants}\label{sec:invariants}
One goal in the study of the (MIP) is to determine properties of $G$, the so-called \emph{group theoretical invariants}, that can be derived from the structure of $kG$ as an algebra. More precisely,  if $G$ and $H$ are $p$-groups satisfying $kG \cong kH$ and $\cor{P}$ is a group theoretical invariant, then both $G$ and $H$ will satisfy $\cor{P}$. The order of $G$ is an obvious invariant, being the same as the $k$-dimension of $kG$. We will usually assume this invariant implicitly. Other invariants are collected in the following list. Many of them are classical, have been known for decades, and are included in \cite[Section 6]{San84}. We only provide references for the ones not included in Sandling's  survey.

\begin{enumerate}[label=$(\arabic*)$]
\item The isomorphism type of the quotient $G/\gamma_2(G)$.
\item The isomorphism type of $G/\Phi(G)$ and hence the minimal number of generators of $G$.
\item The isomorphism type of the quotient $D_n(G)/D_{n+1}(G)$ for any positive integer $n$.
\item The isomorphism type of the quotient $D_n(\gamma_2(G))/D_{n+1}(\gamma_2(G))$ for any positive integer $n$. In particular, if $\gamma_2(G)$ is abelian, then the isomorphism type of $\gamma_2(G)$ is an invariant.
\item The isomorphism type of $\Gamma(G) = \gamma_2(G) \cap \ZG(G)$.
\item The isomorphism type of the quotient $G/\gamma_2(G)^p\gamma_3(G)$ -- the \emph{Sandling quotient}.
\item The isomorphism type of the restricted $k$-Lie algebra $\Jen(G)$ of $G$. % \cite[Corollary 4.1]{Qui68}, \cite[Section 1.3]{HS06}.
\item If the nilpotency class of $G$ equals $2$, then the nilpotency class is an invariant \cite[Theorem 2]{BK07}.
\item If $G$ is $2$-generated, then the isomorphism type of $G/\gamma_2(G)^p\gamma_4(G)$ is an invariant \cite{MM20}. 
\item If $G/C_G(\gamma_2(G)/\Phi(\gamma_2(G))$ is cyclic, then the isomorphism type of $C_G(\gamma_2(G)/\Phi(\gamma_2(G))$ is an invariant \cite[Corollary 7]{Bag99}.
\end{enumerate}

\noindent
We remark that not all authors write about group theoretic invariants: words like  \emph{determined} often serve the same purpose. 
The list of invariants provided above is not complete, but only contains those we will explicitly use in this article. Sometimes, in examples, we will speak of groups sharing \emph{all known group theoretical invariants}. By this we mean the list of invariants included in \cite{MM20}, which, to the best of our knowledge, covers all invariants known to this day.

\subsection{The small group algebra}\label{sec:preliminaries-small}

In this section, we define the small group algebra of $kG$ and briefly discuss its history. We present some classical related constructions and list a number of known properties. This section mostly follows  \cite[Section 2.3]{HS06} and \cite{San89}, though the notation might differ. 

Until the end of the present section, \underline{assume} that $\gamma_2(G)^p\gamma_4(G)=1$ and write $G^{\ab}=G/\gamma_2(G)$ for the abelianization of $G$. 
The \emph{small group algebra} of $kG$
is $kG/I(kG)I(k\gamma_2(G))$ and its group of normalized units is 
$S=S_G=V(kG/I(kG)I(k\gamma_2(G)))$. 
We then have a natural short exact sequence of groups
\begin{equation}\label{eq:SES}
1 \longrightarrow \ls{2}{G}\longrightarrow S\overset{\kappa}{\longrightarrow} V(kG^{\mathrm{ab}})\longrightarrow 1
\end{equation}
with the property that $\kappa^{-1}(G^{\ab})$ is a copy of $G$. Without loss of generality we identify $G$ and $\kappa^{-1}(G^{\ab})$. The group $V(kG^{\mathrm{ab}})$ being abelian, we have that $G$ is a normal subgroup of $S$. 

The history of the small group algebra, or small group ring in this case, dates back to Whitcomb's proof of the fact that integral group rings determine finite metabelian groups \cite[Theorem 2.5]{Pas77}. The small group algebra gains much popularity thanks to \cite{San89} and the determination of Sandling's invariant; see also Section \ref{sec:invariants}. The small group algebra also plays a fundamental role in the proof, by Salim and Sandling, that the (MIP) has a positive answer for groups of order $p^5$ \cite{SS96MC, SS95, SS96p5}; for more detail see also \cite{Sal93}. In the context of central Frattini extensions and with a relative definition of small group algebra, Hertweck and Soriano classify the central Frattini extensions giving rise to isomorphic small group algebras \cite{HS07}. 

%History of the small group ring: 
%\begin{itemize}
%\item Integral case and metabelian groups -- for this see \cite[Theorem 2.5]{Pas77} which is a the theorem, of Whitcomb and Jackson, stating that the integral group ring determines metabelian groups
%\item \cite{San89}. Apparently the result of the paper was clear to Sandling before, as it is already mentioned in the survey (6.25). 
%\item Salim-Sandling \cite{SS96MC}, \cite{SS95}, \cite{SS96p5}
%\item Hertweck-Soriano \cite{HS06}, \cite{HS07}
%\end{itemize}

We conclude the current section by giving the construction of a complement of $G$ in $S$ and deriving related properties. For this, fix $x_1,\ldots,x_n\in G$ (a minimal set of) generators of $G$ with the property that 
$$G^{\mathrm{ab}}=\bigoplus_{i=1}^n\gen{x_i\gamma_2(G)}.$$ 
Denote ${\bf x}=(x_1,\ldots, x_n)$ and, for each $\delta=(\delta_1,\ldots,\delta_n)\in\Z^n$, write ${\bf x}^{\delta}=\overline{x_1}^{\delta_1}\cdots\overline{x_n}^{\delta_n}$. Set $|x_i\gamma_2(G)|=p^{\lambda_i}$. Define, additionally, $D(G)$ as %\ChMima{Check whether you can include $\sum\geq 2$ already here to save space later}
\begin{equation}\label{eq:defD(G)}
D(G)=\left\{\delta=(\delta_1,\ldots, \delta_n) \mid 0\leq \delta_i<p^{\lambda_i}, {\bf\delta}\not\equiv 0\bmod p, \sum_{i=1}^n\delta_i\geq 2\right\}.
\end{equation}
We then have that the subgroup $A=A(G|{\bf x})$ of $S$ that is generated by the elements
\begin{equation}\label{eq:a}
a=1+{\bf x^\delta}=1+\overline{x_1}^{\delta_1}\overline{x_2}^{\delta_2}\cdots\overline{x_n}^{\delta_n} \textup{ with } \delta\in D(G)
\end{equation}
is a complement of $G$ in $S$, equivalently
\[S= G\rtimes A;\]
see for example \cite[Section 2.3]{HS06}.
Moreover, for $x\in G$ and $a$ of the above form, the commutator $[x,a]$ is equal to 
\[
[x,a]=[\ldots[x,\underbrace{ x_1],\ldots x_1}_{\delta_1}],\underbrace{x_2],\ldots, x_2}_{\delta_2}],\ldots x_{n-1}],\underbrace{x_n],\ldots, x_n}_{\delta_n}].
\]
Observe that $[G,A]$ is contained in $\gamma_3(G)$ and, since $G$ has class at most $3$, all elements $a$ of the form \eqref{eq:a} with $\sum_{i=1}^n\delta_i\geq 3$ centralize $G$ and are thus central in $S$. Moreover, $A$ is abelian and $[A,\gamma_2(G)]=1$. %To compute the centralizer of $A$ in $G$ it thus suffices to determine which elements $x\in G$ satisfy
%\[
%1=[x,a] \textup{ for all } a=1+\overline{x_i}\,\overline{x_j} \textup{  with } i\leq j.
%\]
%We now choose a slightly different formulation, though equivalent, to bring together the notations from \cite[Sec.\ 2.3]{HS06} and \cite{San89}. 
%Following the notation from \cite{San89}, let the ideal $J$ of $kG$ be $J=I(kG)I(k\gamma_2(G))+I(k\gamma_2(G))I(kG)$. 
%For this, we define the ideal
%$J=J(G)=II_2+I_2I$ of $R$ and, additionally, the subgroup $W(G)=A(1+\overline{J})$ of $S$. It follows, under our assumptions, that  
%\begin{itemize}
%\item $W(G)=S_3A$ and $S=GW(G)$,
%\item $G\cap W(G)=S_3$,
%\item $W(G)\cap (1+I_2R)=1+J$,
%\item $\kappa(W(G))=W(G^{\mathrm{ab}})=\kappa(A)$.
%\end{itemize} 
%In particular, the projection $V(R)\rightarrow V(kG^{\mathrm{ab}})$ induces a short exact sequence
%\[
%1\rightarrow G\rightarrow V(R)/(1+II_2)\cong S\rightarrow W(G^{\mathrm{ab}})\rightarrow 1.
%\]

The following proposition collects some results, already presented by Salim in his PhD thesis \cite{Sal93}. Point (1) is a special instance of \cite[Thm.\ 3.5]{Sal93}, point (4) follows from \cite[Thm.\ 3.7]{Sal93}, while point (3) is included in \cite[Cor.\ 3.11]{Sal93}. 

\begin{proposition}\label{prop:propsS-gp}
The following hold:
\begin{enumerate}[label=$(\arabic*)$]
\item for each $i\in\Z_{\geq 2}$, one has $\gamma_i(S)=\gamma_i(G)$,
\item $\Gamma(S)=\Gamma(G)$,
\item $[A^p,S]=\graffe{1}$,% -- so this is an improvement of the above lemma from \cite{HS/06},
\item $G\cap\ZG(S)=\ZG(G)$.
%\item $\Cyc_S(A)= \Cyc_G(A)\times A$.
\end{enumerate}
\end{proposition}

\begin{proof}
We prove (1) and (2) together. % We will show that $\gamma_2(S)=\gamma_2(G)$, $\gamma_3(S)=\gamma_3(G)$, and $\gamma_4(S)=1$. 
Let $x,y\in S$ and let $g,h\in G$ and $a,b\in A$ be such that $x=ga$ and $y=hb$. It follows from Lemma \ref{lemma:comm-formulas} and the fact that $[G,\gamma_3(G)]=1=[A,A\gamma_2(G)]$ that
\begin{align*}
[x,y] &= [ga,hb]=[ga,b][ga,h]^b\\
         &= [g,b]^a[a,b]([g,h]^a[a,h])^b\\
         &= [g,b][g,h][a,h] \in \gamma_3(G)\gamma_2(G)\gamma_3(G)=\gamma_2(G).
\end{align*}
Since $\gamma_2(G)\subseteq \gamma_2(S)$, we derive that in fact $\gamma_2(S)=\gamma_2(G)$. 

We now show that $\Gamma(S)=\Gamma(G)$. Since $\ZG(S)\cap G\subseteq \ZG(G)$, we have that $\ZG(S)\cap \gamma_2(S)\subseteq \ZG(G)\cap \gamma_2(G)$. To show the opposite inclusion, it now suffices to show that $\ZG(G)\cap \gamma_2(G)\subseteq \ZG(S)$. To this end, let $x=ga\in S$ and let $z\in\Gamma(G)$. Then 
$[z,x]=[z,ga]=[z,a][z,g]^a=1$,
because $[z,a]\in [\gamma_2(G),A]=1$ and $[z,g]\in[\ZG(G),G]=1$, and therefore $z\in\ZG(S)$. 

To show that $\gamma_3(S)=\gamma_3(G)$, observe that $\gamma_2(S)A=\gamma_2(G)A\subseteq \Cyc_S(\gamma_2(G))=\Cyc_S(\gamma_2(S))$. We then get the following commutative diagram of natural maps
\[
\xymatrix{
S/\gamma_2(S)A\otimes \gamma_2(S)/\Gamma(S) \ar[d]_{\cong}\ar@{>>}[r] & \gamma_3(S) \\
G/\gamma_2(G)\otimes \gamma_2(G)/\Gamma(G)\ar@{>>}[r] & \gamma_3(G) \ar[u]_{\subseteq}
}
\]
and so the inclusion map is an equality.

We now show that $\gamma_4(S)=1$. Since $A$ is abelian and $S=G\rtimes A$, we have that $[A,S]\subseteq \gamma_3(S)=\gamma_3(G)$. It follows that
$\gamma_3(S)\subseteq \gamma_2(G)\cap \ZG(G) = \gamma_2(S)\cap \ZG(S)$
and, in particular, $\gamma_3(S)$ is central in $S$, equivalently $\gamma_4(S)=[S,\gamma_3(S)]=1$.

(3) This is a direct consequence of $\gamma_2(G)$ being elementary abelian. %It follows from the definition of $A$ and (4) that the commutator map on $S$ induces a bilinear map $S\times A\longrightarrow G_3$ which then factors through
%\[
%S/G_2A\times A/A^p\longrightarrow G_3
%\]
%because $G_3$ is elementary abelian. 

(4) We clearly have that $\ZG(S)\cap G\subseteq \ZG(G)$ and so we prove the other inclusion. To this end, let $z\in\ZG(G)$ and $s=ga\in S$. It follows from Lemma~\ref{lemma:comm-formulas} that 
$[z,ga]=[z,a][z,g]^a=[z,a]$.
If we now take $a=1+\overline{x_i}\,\overline{x_j}$ as in \eqref{eq:a}, then $[z,a]=[[z,x_i],x_j]=1$ and we are done.
%(7) We have that 
%$A\Cyc_G(A)=A(\Cyc_S(A)\cap G)=(GA)\cap\Cyc_S(A)=S\cap \Cyc_S(A)=\Cyc_S(A)$ and conclude observing that
%$A\cap \Cyc_G(A)=A\cap G\cap\Cyc_S(A)=\graffe{1}$.
\end{proof}

\begin{lemma}\label{lemma:Z2}
The following hold: 
\[ \ZG_2(S)=\ZG_2(G)A, \quad [\ZG_2(G),A]=1, \quad \ZG_2(G)=\ZG_2(S)\cap G.\]
Moreover, $\ZG_2(S)$ is equal to the direct product of $\ZG_2(G)$ with $A$.
\end{lemma}

\begin{proof}
Note that, for $x=ga$ and $y=hb$ in $S$, the following equality holds $[x,y]=[ga,hb]=[g,b][g,h][a,h]$ and, moreover, that, by Proposition \ref{prop:propsS-gp}(2), we have
\begin{align*}
\ZG_2(G) & = \graffe{z\in G \mid \textup{ for all } g\in G \ : \ [z,g]\in \Gamma(G)},\\
\ZG_2(S) & = \graffe{z\in S \mid \textup{ for all } s\in S \ : \ [z,s]\in \Gamma(S)=\Gamma(G)}.
\end{align*}
We start by showing that $\ZG_2(G)A$ is contained in $\ZG_2(S)$. First, observe that $A$ is contained in $\ZG_2(S)$ because $[S,A]=[G,A]\subseteq \gamma_3(G)\subseteq \Gamma(G)$. Moreover, if we now take $a=1$ and $x=g\in\ZG_2(G)$, it follows from the above formula that 
$[g,hb]=[g,b][g,h]\in \gamma_3(G)\Gamma(G)=\Gamma(G)$.
We now show that $\ZG_2(S)$ is contained in $\ZG_2(G)A$. To this end, assume that $x=ga\in\ZG_2(S)$: we show that $g\in\ZG_2(G)$. Indeed, it follows from Lemma \ref{lemma:comm-formulas} that 
$[g,h]=[g,b]^{-1}[x,y][a,h]^{-1}\in \gamma_3(G)\Gamma(G)=\Gamma(G)$.

We now show that $[\ZG_2(S),A]=1$. For this, observe that 
$[\ZG_2(S),A]=[\ZG_2(G),A]$. Let now $x\in\ZG_2(G)$ and $a=1+\overline{x_i}\,\overline{x_j}\in A$ as in \eqref{eq:a}. Then we have 
$[x,a]=[[x,x_i],x_j]\in[\Gamma(G),G]=1$
and thus the claim is proven.

We conclude by observing that $\ZG_2(S)\cap G=(\ZG_2(G)A)\cap G=(G\cap A)\ZG_2(G)=\ZG_2(G)$.
\end{proof}

\section{Applications of the small group algebra}\label{sec:SmallGroupRing}

In this section we will push techniques of Salim and Sandling further to achieve a positive solution of the (MIP) for some families of groups, yielding in particular new results for orders $p^6$ and $p^7$. We do this with the aid of the small group algebra, as introduced in Section \ref{sec:preliminaries-small}.

Until the end of the present section, \underline{assume} that $\gamma_2(G)^p\gamma_4(G)=1$. It is not difficult to show that, under the additional assumption that $p$ is odd, the subgroup structure of $G$ is as in Figure~\ref{fig:subgroupStructure}. We remark that, in Figure~\ref{fig:subgroupStructure}, the only inclusion requiring $p$ odd is $G^p\subseteq \ZG(G)$. In the statements in which we will require this, we will explicitly assume that $p$ is odd.

Write $G^{\mathrm{ab}}=G/\gamma_2(G)$ for the abelianization of $G$ and $S=V(kG/I(kG)I(k\gamma_2(G)))$ for the normalized units of the small group algebra of $kG$. Let, moreover, ${\bf x}=(x_1,\ldots, x_n)$ be a vector of generating elements of $G$ satisfying $G^{\mathrm{ab}}=\oplus_{i=1}^n\gen{x_i\gamma_2(G)}$ and denote $|x_i\gamma_2(G)|=p^{\lambda_i}$. Let $A=A(G | {\bf x})$ denote the complement of $G$ in $S$ as defined in Section \ref{sec:preliminaries-small}.  Then $A$ satisfies the following:
\[
S=G\rtimes A, \quad [G,A]\subseteq \gamma_3(G), \quad [A,A\gamma_2(G)]=1.
\]
We will show in Section \ref{sec:complement} that any group base of $kG$ can be complemented by $A$ in $S$. Moreover, we will show how generating sets of two group bases can be related via translation by elements of $A$. Section \ref{sec:main-small} collects the main results of Section \ref{sec:SmallGroupRing}. We present implications of our main results in Section \ref{sec:cor-small}, with a particular emphasis on the case of groups of order $p^6$ and $p^7$. Additionally, we give computational evidence aimed at depicting our contribution for the last groups and give a critical analysis of the limits of our methods for order $p^6$.

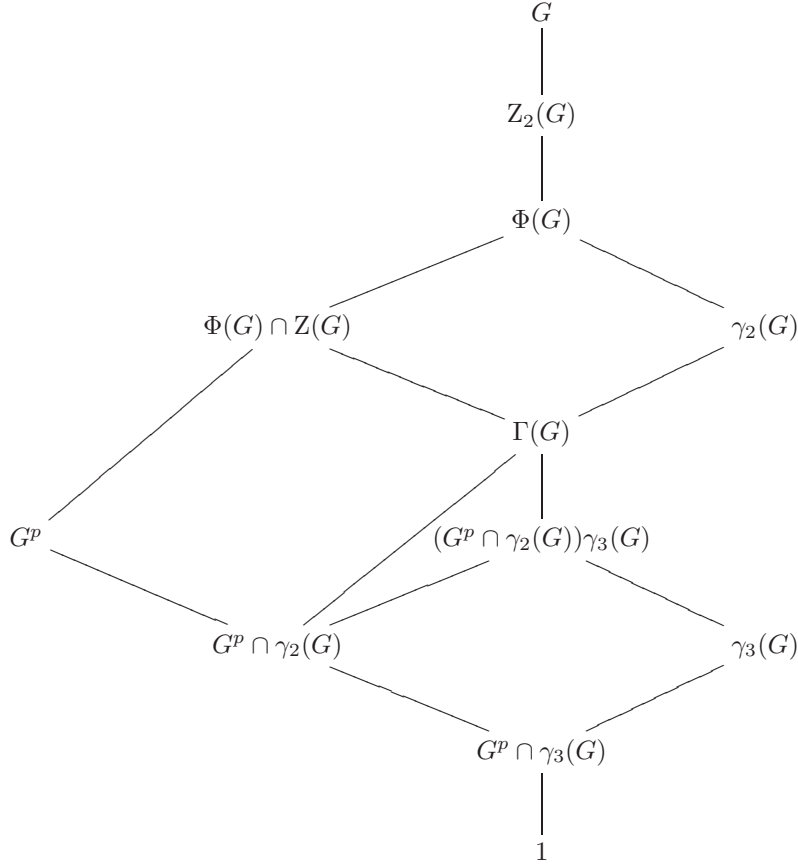
\begin{figure}[H]
\begin{equation*}
\xymatrix{
&& &  G \ar@{-}[d]&& \\
&& &  \ZG_2(G) \ar@{-}[d]&& \\
&& & \ar@{-}[dl] \Phi(G) \ar@{-}[dr]&& \\
&& \ar@{-}[ddll] \Phi(G)\cap\ZG(G) \ar@{-}[dr] & & \ar@{-}[dl] \gamma_2(G) &  \\
& &  & \ar@{-}[ddl] \Gamma(G) \ar@{-}[d] &  \\
G^p \ar@{-}[drr] && & \ar@{-}[dl] (G^p\cap\gamma_2(G))\gamma_3(G) \ar@{-}[dr] & \\
&& \ar@{-}[dr] G^p\cap \gamma_2(G) & &\ar@{-}[dl] \gamma_3(G) \\
&& & G^p\cap\gamma_3(G)\ar@{-}[d] & \\
&& & 1&
}
\end{equation*}
{\caption{Subgroup structure of $G$ in case $\gamma_2(G)^p\gamma_4(G) =1$ and $p$ odd}\label{fig:subgroupStructure}}
\end{figure}

\subsection{Compatible complements}\label{sec:complement}

\noindent
The aim of the current section is to relate different group bases of $kG$ within $S$. We observe that, if $H$ is another group base of $kG$, then $kG^{\ab}=kH^{\ab}$ and so, through the short exact sequence in \eqref{eq:SES}, the group $S$ possesses a normal subgroup that is isomorphic to $H$. Because of this, we can consider every group base of $kG$ to be embedded in $S$ as a normal subgroup. 

\begin{proposition}\label{prop:sameA}
Let $H$ be a group base of $kG$ and consider it embedded in $S$.
Then $S=G\rtimes A=H\rtimes A$.
\end{proposition}

\begin{proof}
Write ${\bf y}=(y_1,\ldots,y_n)=(x_1\gamma_2(G),\ldots, x_n\gamma_2(G))$ and denote $A^{\ab}=A(G^{\ab}|{\bf y})$. 
Let $D(G)$ be as defined in \eqref{eq:defD(G)} and define $\cor{B}  = \graffe{{\bf y^\delta} \mid \delta\in D(G)}$.
%\begin{align*}
%\cor{B}  = \graffe{{\bf y^\delta} \mid \delta\in D(G)} \ \textup{ and } \
%\cor{B}^{\circ} = \graffe{b\in \cor{B}\setminus\graffe{1}\mid 1+b\notin G}.
%\end{align*}
It follows from the description in Section \ref{sec:preliminaries-small} that $A^{\ab}$ is generated by $1+\cor{B}$ and, with the notation from \eqref{eq:SES}, that $A^{\ab}=\kappa(A)$. Moreover, $V(kG^{\ab})$ is the direct product of $G^{\ab}$ and $A^{\ab}$. 
Define now $U$  to be the $k$-linear span of $\cor{B}$ in %$kG$ and by $U^{\mathrm{ab}}$ the image of $U$ in 
$kG^{\mathrm{ab}}$: 
we claim that %$U$ is an ideal of $kG^{\mathrm{ab}}$ and that 
$1+U=A^{\ab}$. %For this, write $y_i=x_i\gamma_2(G)$ and observe that $\graffe{y_i}_{i=1}^n$ satisfies $G^{\mathrm{ab}}=\oplus_{i=1}^n\gen{y_i}$. 
For this, observe that, as we have defined it, $U$  satisfies the following equality for all positive integers $m$:
\[
I(kG^{\mathrm{ab}})^m/I(kG^{\mathrm{ab}})^{m+1}=H_m(kG^{\mathrm{ab}})/I(kG^{\ab})^{m+1}\oplus ((U\cap I(kG^{\ab})^m)+I(kG^{\ab})^{m+1})/I(kG^{\ab})^{m+1}.
\]
In particular, it holds that 
\[
|U|=|I(kG^{\ab})|/\prod_{m\geq1}|D_m(G^{\ab}):D_{m+1}(G^{\ab})|=|I(kG^{\ab})|/|G|=|A^{\ab}|.
\]
Since $U$ is multiplicatively closed, we get that $1+U$ is a group and thus $\gen{1+\cor{B}}=1+U$.
It follows from the theory of Zassenhaus ideals, cf.\ Section~\ref{sec:Ideals}, that $A^{\ab}=1+U$ and therefore 
$$V(kG^{\ab})=H^{\ab}\times A^{\ab}=\kappa(H)\times \kappa(A).$$
To conclude, assume, for a contradiction, that $A$ intersects $H$ non-trivially. Thanks to Proposition \ref{prop:propsS-gp}(1), the subgroups $\gamma_2(H)$ and $\gamma_2(G)$ coincide and thus $A\cap \gamma_2(H)=1$. It follows then that $H\cap A$ has a non-trivial image in $V(kG^{\ab})$, but this contradicts the fact that $V(kG^{\ab})= H^{\ab}\times \kappa(A)$.
\end{proof}

\noindent
Thanks to the last result, the subgroup $A$ is a mutual complement of both $G$ and $H$. It follows that, if $S$ determines the isomorphism type of $G$, then so does $S/(A\cap\ZG(S))$. Without loss of generality, we \underline{assume}, until the end of Section~\ref{sec:SmallGroupRing}, that $A\cap\ZG(S)=1$. It follows from Proposition \ref{prop:propsS-gp}(3) that $A$ is elementary abelian. 

\begin{lemma}\label{lemma:gensG->H}
Let $H$ be a group base of $kG$ and consider it embedded in $S$. Let $g_1,\ldots,g_n$ be a generating set of $G$. Then there exist $a_1,\ldots,a_n \in A$ such that %, for $1 \leq i \leq n$, one has $g_ia_i \in H$  and 
$g_1 a_1,\ldots, g_n a_n$ is a generating set of $H$.  
\end{lemma}

\begin{proof}
We start by claiming that, if $\graffe{y_i\mid 1\leq i\leq m}$ is a generating set of $G$, where each $y_i$ is written as $y_i=z_ib_i$ with $z_i\in H$ and $b_i\in  A$, then $H$ is generated by $z_1,\ldots,z_m$.
To this end, call $Z=\gen{z_i \mid 1\leq i\leq m}$: we show that $Z=H$. First observe that $\gamma_3(H)\subseteq Z$: indeed, for any choice of $i,j,\ell$, one has $[y_i,y_j,y_\ell]=[z_i,z_j,z_\ell]$. Now, $H$ being a group base of $kG$, by Proposition \ref{prop:propsS-gp}(1) the subgroups $\gamma_3(H)$ and $\gamma_3(S)$ coincide and thus it follows that $Z$ contains $\gamma_3(S)$. Moreover, since $[S,A]$ is contained in $\gamma_3(S)$, we derive that actually $Z$ contains $\gamma_2(S)$. In particular, $Z$ is normal in $H$ and so Proposition \ref{prop:sameA} yields $S=HA\subseteq ZA$ and thus $Z=H$.

Write now each $g_i$ as $g_i=h_ic_i$, where $h_i\in H$ and $c_i\in A$; this can be done thanks to Proposition \ref{prop:sameA}. A consequence of the above claim is that $h_1,\ldots,h_n$ generate $H$. We conclude by defining $a_i=c_i^{-1}$.
\end{proof}

\subsection{Main results}\label{sec:main-small}

\noindent
In the present section, we leverage results from Section \ref{sec:complement} to solve the (MIP) for certain families of groups. Recall that, without loss of generality, we have assumed in Section~\ref{sec:complement} that $A\cap\ZG(S)=1$. The following is our first main result.

\begin{theorem}\label{th:MaxAbelianCentDerSubg}
Assume that $p$ is odd, $\gamma_2(G)^p\gamma_4(G)=1$, and $\Cyc_G(\gamma_2(G))$ is a maximal subgroup of $G$ that is abelian. Then the small group algebra of $kG$ determines $G$ up to isomorphism. Moreover, the \emph{(MIP)} has a positive answer for $G$.
\end{theorem}

\begin{proof}
Set $C=\Cyc_G(\gamma_2(G))$ and let $H$ be another group base of $kG$, which we consider to be embedded in the group $S$ of normalized units of the small group algebra of $kG$. Then Proposition \ref{prop:sameA} yields that $S=G\rtimes A=H\rtimes A$. Now, by Proposition \ref{prop:propsS-gp}(1), we have $\gamma_2(G)=\gamma_2(H)$, from which we derive that $\Cyc_S(\gamma_2(S))=CA=\Cyc_H(\gamma_2(H))A$. We will show that, for a certain family of generators $g_1,\ldots,g_n$ of $G$, there exist generating elements $h_1,\ldots, h_n$ of $H$   satisfying the same relations as $g_1,\ldots, g_n$. 
To this end, let $g_1,\ldots,g_n$ be a minimal generating set of $G$ satisfying the following: 
\begin{itemize}
\item $C=\gen{g_2,\ldots,g_n}\gamma_2(G)$,
\item for $|\gamma_2(G):\gamma_2(G)\cap G^p|=p^{r-1}$, the elements $g_2,\ldots, g_r$ are such that $[g_2,g_1],\ldots,[g_{r},g_1]$ form a basis of $\gamma_2(G)$ modulo $\gamma_2(G)\cap G^p$.
\end{itemize}
The following ensures the existence of such a generating set.
Since $C$ is abelian, $\gamma_2(G)=[\gen{g_1},C]$ and, since $[C,\gamma_2(G)]=1$, we have that $\gamma_3(G)=[\gen{g_1},\gamma_2(G)]$.
Let now $a_1,\ldots, a_n$ be as in Lemma \ref{lemma:gensG->H} and observe that $C_H=\Cyc_H(\gamma_2(H))$ is generated by $g_2a_2,\ldots, g_n a_n$ modulo $\gamma_2(H)$.
Thanks to Bagi\'nski's invariant from Section \ref{sec:invariants}(9), the groups $C$ and $C_H$ are maximal subgroups respectively in $G$ and $H$ and both of them are abelian. In other words, for each $i,j\in\graffe{2,\ldots, n}$, it holds that $[g_ia_i,g_ja_j]=[g_i,g_j]=1$. 
Denote by ${\bf g}$ the vector $(g_1,\ldots,g_n)$ and, for each $\delta=(\delta_1,\ldots, \delta_n)\in\Z^n$, write
${\bf g}^{\delta}=g_1^{\delta_1}\ldots g_n^{\delta_n}.$
For each $i\in\graffe{1,\ldots, n}$, we now choose $\beta_i\in \gamma_2(G)=\gamma_2(H)$ such that it satisfies $[\beta_i,g_1]=%\begin{cases}
%1 & \textup{if } i\leq r \\
([a_i,g_1][g_i,a_1])^{-1}$.
Define, moreover, for each $i$, the element $h_i$ to be $g_i\beta_ia_i$ and ${\bf h}=(h_1,\ldots, h_n)$. Then, for $i,j,\ell\in\graffe{1,\ldots, n}$ the following hold: 
\begin{itemize}
\item for $i,j\geq 2$, one has $[h_i,h_j]=[g_i,g_j]$ because $A\gamma_2(G)$ is abelian and $[g_ia_i,g_ja_j]=1$;
\item $[h_i,h_1]=[g_i,g_1]$ thanks to the definition of the $\beta_i$'s:
\begin{align*}
[h_i,h_1] & = [g_i\beta_ia_i,g_1a_1] = [g_i\beta_ia_i, g_1][g_i,a_1] =[g_i,g_1][\beta_ia_i,g_1][g_i,a_1]\\
& = [g_i,g_1][\beta_i,g_1][a_i,g_1][g_i,a_1]=[g_i,g_1];
\end{align*}
\item $[h_i,h_j,h_\ell]=[g_i,g_j,g_\ell]$ because $[G,A]\subseteq\gamma_3(G)\subseteq\ZG(G)$;
\item for each $\delta\in p\Z^n$, one has ${\bf h}^{\delta}={\bf g}^{\delta}$ because $[A\gamma_2(G),G]$ is central in $S$, $\gamma_2(G)$ is elementary abelian, and $p$ is odd.
\end{itemize}
This shows that $G$ and $H$ have the same presentations in terms of $g_1,\ldots,g_n$ and $h_1,\ldots, h_n$, respectively. So, in particular, $G$ and $H$ are isomorphic. 
\end{proof}

\noindent
The small group algebra has been used to (positively) solve the (MIP) for $2$-generated groups $G$ satisfying $\gamma_2(G)^p\gamma_4(G) = 1$ \cite{MM20}. Extending the last result to $3$-generated groups is not straightforward; Theorem~\ref{th:K_G} is a first step in this direction. 

\begin{lemma}\label{lemma:K_G}
Assume that $G$ has class $3$ and $|G:\Phi(G)|=|G:\ZG_2(G)|=p^3$. %. Assume, additionally, that $\Phi(G)=\ZG_2(G)$ and $\Phi(G)$ has index $p^3$ in $G$. 
Then the following hold: 
\begin{enumerate}[label=$(\arabic*)$]
\item $|\ls{2}{G}:\ls{3}{G}|\leq p^3$, 
\item there exists a unique subgroup $K_G$ of $G$, containing $\ZG_2(G)$, with the property that the following sequence induced by the commutator map
is exact:
\[
1\rightarrow \wedge^2 (K_G/\ZG_2(G))\rightarrow \wedge^2(G/\ZG_2(G))\rightarrow \gamma_2(G)/\Gamma(G)\rightarrow 1.
\]
\item if $|\ls{2}{G}:\Gamma(G)|=p^3$, then $K_G=\ZG_2(G)$, 
\item if $|\ls{2}{G}:\Gamma(G)|<p^3$, then $|\ls{2}{G}:\Gamma(G)|=p^2$ and $K_G$ is maximal in $G$.
\end{enumerate}
\end{lemma}

\begin{proof}
It follows from the assumptions that $\Phi(G)=\ZG_2(G)$. 
Now, to lighten the notation, we write $U=G/\Phi(G)$, which is a $3$-dimensional vector space over $k$.  Since $\gamma_2(G)$ is elementary abelian, the commutator map induces a bilinear map $U\times U\rightarrow \gamma_2(G)/\gamma_3(G)$. The last map being alternating, the universal property of wedge products induces a surjective linear map $\wedge^2U\rightarrow \gamma_2(G)/\gamma_3(G)$ and, as $\wedge^2U$ has dimension $3$, the first claim follows.

We now prove (2)-(3)-(4) simultaneously. For this, note that $\gamma_3(G)\subseteq \Gamma(G)$ and, since $\Phi(G)=\ZG_2(G)$, the commutator map induces a non-degenerate alternating map $\nu:U\times U\rightarrow\gamma_2(G)/\Gamma(G)$. In particular, the dimension of $U$ being odd, we get that $$p^2\leq |\ls{2}{G}:\Gamma(G)|\leq |\gamma_2(G):\gamma_3(G)|\leq p^3.$$
Let now $W$ be a subspace of $U$ such that the following short exact sequence, induced by $\nu$, is exact:
\[
1\rightarrow \wedge^2W\longrightarrow\wedge^2U\overset{\tilde{\nu}}{\longrightarrow} \gamma_2(G)/\Gamma(G)\rightarrow 1.
\]
The existence of $W$ is guaranteed by the following considerations. The dimension of $\wedge^2U$ being $3$, one has that $|\gamma_2(G):\Gamma(G)|=p^3$ implies that $\tilde{\nu}$ is injective and $W=0$. 
%If $\gamma_3(G)=\Gamma(G)$, then $\tilde{\nu}$ is injective and $W=0$; 
In particular, in the last case, $W$ is unique with this property and $K_G=\ZG_2(G)$. Assume now that $|\gamma_2(G):\Gamma(G)|=p^2$. 
Then $\tilde{\nu}$ has a nontrivial kernel and, $\tilde{\nu}$ being induced by $\nu$, there is a subspace $W$ of $U$ such that $\ker\tilde{\nu}=\wedge^2W$. The map $\tilde{\nu}$ being surjective, it follows that $W$ has dimension $2$ and the fact that $|\ls{2}{G}:\Gamma(G)|=p^2$ ensures that $W$ is unique with this property. In particular, $K_G$, defined by $K_G/\ZG_2(G)=W$, is maximal.
\end{proof}

%\begin{definition}
%Assume that $\Phi(G)=\ZG_2(G)$ and that $\Phi(G)$ has index $p^3$ in $G$.  Then the \emph{Manneken of} $G$ is the subgroup $K_G$ from Lemma $\ref{lemma:K_G}$. \ChLeo{The name was funny, but we never used it. Should we make this just a \textbf{Notation} and carry on without the name?}
%\end{definition}

\noindent
In the remaining part of Section~\ref{sec:SmallGroupRing}, let $K_G$ denote the characteristic subgroup from Lemma \ref{lemma:K_G}.

\begin{theorem}\label{th:K_G}
Assume that $p$ is odd, 
$\gamma_2(G)^p\gamma_4(G)=1$, and $|G:\Phi(G)|=|G:\ZG_2(G)| = p^3$. If $G$ has class $3$, assume, furthermore, that $[[K_G, G], G] \subseteq [K_G, \gamma_2(G)]$. Then the isomorphism type of $G$ is determined by the unit group of the small group algebra of $kG$. Moreover, the \emph{(MIP)} has a positive answer for $G$.
\end{theorem}

\begin{proof}
If $\gamma_3(G) = 1$, then the result follows from the work of Sandling, cf.\ Section \ref{sec:invariants}(6). Assume now that $\gamma_3(G) \neq 1$, so it follows from the assumptions that $\Phi(G)=\ZG_2(G)$.
Similarly as in the proof of Theorem~\ref{th:MaxAbelianCentDerSubg}, let $H$ be another group base of $kG$, which we consider to be embedded in the group $S$ of normalized units of the small group algebra of $kG$. Then the equality $S=G\rtimes A=H\rtimes A$ holds by Proposition \ref{prop:sameA} and $\gamma_2(G)=\gamma_2(H)$, by Proposition \ref{prop:propsS-gp}(1). 
We will show that, for a certain family of generators $g_1,g_2,g_3$ of $G$, there exist generating elements $h_1,h_2, h_3$ of $H$   satisfying the same relations as $g_1,g_2, g_3$.  We will study two cases separately.

First, suppose that $|\ls{2}{G}:\Gamma(G)| = p^3$ and let $g_1, g_2, g_3$ be any minimal generating set of $G$. By Lemma~\ref{lemma:gensG->H} there are elements $a_1,a_2,a_3$ in $A$ such that $h_i = g_ia_i$ generate $H$. We will show that $h_1,h_2,h_3$ already satisfy the same relations as $g_1,g_2,g_3$. From Lemma~\ref{lemma:K_G} we know that $K_G = Z_2(G)$, $\Gamma(G) = \gamma_3(G)$, and the commutators $[g_2,g_1], [g_3,g_1], [g_3, g_2]$ form a basis of $\gamma_2(G)$ modulo $\gamma_3(G)$ and they do not lie in $G^p$. Notice that for $i,j,\ell \in \{1,2,3 \}$ and $m$ any positive integer we have 
\begin{itemize}
\item $[g_i, g_j] \equiv [h_i,h_j] \mod \gamma_3(G)$ because $[G,A]\subseteq \gamma_3(G)$;
\item $[g_i,g_j,g_\ell] = [h_i,h_j,h_\ell]$ as $[G,A] \subseteq Z(G)$;
\item $g_i^{p^m} = h_i^{p^m}$ because $A\gamma_2(G)$ is elementary abelian, $[G,A]$ is central, and $p$ is odd.
\end{itemize}
It follows that $h_1,h_2$ and $h_3$ satisfy the same relations as $g_1,g_2$ and $g_3$.

We now consider the case $|\ls{2}{G}:\Gamma(G)| = p^2$, which is the only case left in view of Lemma~\ref{lemma:K_G}. By Lemma~\ref{lemma:K_G}, the subgroup $K_G$ is maximal in $G$. We choose a generating set $g_1,g_2,g_3$ of $G$ such that $g_2,g_3 \in K_G$ and so, by Lemma~\ref{lemma:gensG->H}, there exist $a_1,a_2,a_3 \in A$ such that $g_1a_1, g_2a_2, g_3a_3$ generate $H$. Let $\beta_2$ and $\beta_3$ be elements of $\gamma_2(G)$ such that 
\[[\beta_3, g_2][g_3,\beta_2]=([a_3,g_2][g_3,a_2])^{-1}.\]
 Note that such $\beta_2$ and $\beta_3$ exist because $[a_3,g_2][g_3,a_2] \in [K_G,G,G]\subseteq [K_G,\gamma_2(G)]$ and $g_2, g_3$ form a basis of $K_G$ modulo $\gamma_2(G)\ZG(G)$. Set $h_1 = g_1a_1$, $h_2 = g_2\beta_2a_2$, and $h_3 = g_3\beta_3a_3$. As $\gamma_2(G) = \gamma_2(H)$, we know that $h_2$ and $h_3$ are elements of $H$. Now $[g_2,g_1]$ and $[g_3,g_1]$ form a basis of $\gamma_2(G)$ modulo $\Gamma(G)$ and do not lie in $G^p$. So  for $i,j,\ell \in \{1,2,3 \}$ and $m$ any positive integer we have:
\begin{itemize}
\item $[g_2, g_1] \equiv [h_2,h_1] \bmod \gamma_3(G)$ and $[g_3, g_1] \equiv [h_3,h_1] \bmod \gamma_3(G)$ because $[G,A]\subseteq \gamma_3(G)$;
\item $[h_3,h_2] = [g_3,g_2][\beta_3, g_2][g_3,\beta_2][a_3,g_2][g_3,a_2] = [g_3,g_2]$;
\item $[g_i,g_j,g_\ell] = [h_i,h_j,h_\ell]$ since $[G,A] \subseteq Z(G)$;
\item $g_i^{p^m} = h_i^{p^m}$ because $A\gamma_2(G)$ is elementary abelian, $[G,A]$ is central, and $p$ is odd.
\end{itemize}
This shows that $G$ and $H$ have the same presentations in terms of respectively $g_1,g_2,g_3$ and $h_1,h_2,h_3$  so, in particular, $G$ and $H$ are isomorphic.
\end{proof}

\subsection{Corollaries and further discussion}\label{sec:cor-small}

\noindent
The upcoming series of corollaries of our results from Section~\ref{sec:main-small} yields a number of new classes of $p$-groups of orders $p^6$ and $p^7$ for which the (MIP) is positively solved.

\begin{corollary}\label{cor:gamma3p}
Assume that $p$ is odd, $\gamma_2(G)^p\gamma_4(G)=1$, and $|G:\Phi(G)|=|G:\ZG_2(G)|=p^3$. The isomorphism type of $G$ is determined by the normalized units of the small group algebra of $kG$ if one of the following holds:
\begin{enumerate}[label=$(\arabic*)$]
\item $|\gamma_2(G):\Gamma(G)|=p^3$,
\item $|\gamma_3(G)|=p$.
\end{enumerate}
\end{corollary}

\begin{proof}
We assume that $\gamma_3(G) \neq 1$ as a consequence of Section \ref{sec:invariants}(6). Then, by Lemma \ref{lemma:K_G}, the index $|\gamma_2(G):\Gamma(G)|$ is either $p^2$ or $p^3$.

If $|\gamma_2(G):\Gamma(G)|=p^3$, then Lemma~\ref{lemma:K_G} yields $K_G = Z_2(G)$ and therefore we have $[K_G, G] \subseteq Z(G)$ and $[[K_G,G],G,] = 1$. The claim follows from Theorem~\ref{th:K_G}.

Assume now that $|\gamma_3(G)| = p$ and, as a consequence of (1), that $|\gamma_2(G):\Gamma(G)|=p^2$. Then, by Lemma \ref{lemma:K_G}(4), the subgroup $K_G$ is maximal in $G$. Now the commutator map
\[G/\Cyc_G(\gamma_2(G)) \times \gamma_2(G)/\Gamma(G) \rightarrow \gamma_3(G)\]
is a non-degenerate alternating map and, as $ \gamma_2(G)/\Gamma(G)$ is $2$-dimensional, $G/\Cyc_G(\gamma_2(G))$ has dimension $2$. In particular, $\Cyc_G(\gamma_2(G))$ is different from $K_G$ and thus $[K_G,\gamma_2(G)] \neq 1$. It follows that $[K_G,\gamma_2(G)] = \gamma_3(G)$ and so we are done by Theorem~\ref{th:K_G}.
\end{proof}

\begin{corollary}\label{cor:p6}
Assume that $|G|=p^6$, $\gamma_2(G)^p\gamma_4(G)=1$, and $|G:\ZG_2(G)|=p^3$. Then the isomorphism type of $G$ is determined by the normalized units of the small group algebra of $kG$. 
\end{corollary}

\begin{proof}
The work in \cite{Wur93} covers the case $p=2$ and Section \ref{sec:invariants}(6) covers the case $\gamma_3(G)=1$. We assume that $p$ is odd and that  $\gamma_3(G) \neq 1 $. We show that $\Phi(G)=\ZG_2(G)$. The inclusions $\gamma_2(G)\subseteq\Phi(G)\subseteq\ZG_2(G)$ are clear, so we show that $\ZG_2(G)=\gamma_2(G)$. The commutator map induces a non-degenerate map
\[
G/\ZG_2(G)\times G/\ZG_2(G)\rightarrow \ZG(G)\gamma_2(G)/\ZG(G)
\] 
and so, the dimension of $G/\ZG_2(G)$ being odd, we have that $p<|\ZG(G)\gamma_2(G):\ZG(G)|\leq p^3$. Moreover, the class of $G$ being $3$, we have that $|\gamma_3(G)|\geq p$ and $\gamma_2(G)$ is not central. We derive that $\ZG(G)=\gamma_3(G)$ and $\gamma_3(G)$ has order $p$. In particular, $\gamma_2(G)$ contains $\ZG(G)$ and $\gamma_2(G)=\ZG_2(G)=\Phi(G)$. We conclude thanks to Corollary~\ref{cor:gamma3p}.
\end{proof}

\noindent
We remark that, under the conditions that $G$ has class $3$, order $p^6$, and $\gamma_2(G)$ is elementary abelian, the only possibilities for the index of $\ZG_2(G)$ in $G$ are $p^2$ and $p^3$. The latter case is settled by Corollary \ref{cor:p6} while the case $|G:\ZG_2(G)|=p^2$ cannot be settled within the small group algebra, as Example \ref{ex:553vs554} shows.

\begin{remark}[(Isoclinism classes.)]\label{rmk:isoclinism}
We here shortly discuss the influence  of our results on the (MIP) for groups of order $p^6$ in terms of isoclinism classes. Following P.\ Hall's definition \cite[Definition 4.28]{Suz86}, two groups $X$ and $Y$ are \emph{isoclinic} if there exist isomorphisms 
\[
\varphi:X/\ZG(X)\rightarrow Y/\ZG(Y) \textup{ and } \psi: \gamma_2(X)\rightarrow \gamma_2(Y) 
\]
commuting with the commutator map, in other words such that, for each $x,x'\in X$, one has \
$$\psi([x,x'])=[\varphi(x)\ZG(X),\varphi(x')\ZG(X)].$$
We rely on the classification, given in \cite{Jam80}, of the 43 isoclinism classes $\Phi_1,\ldots,\Phi_{43}$ of groups of order at most $p^6$ for $p$ odd. 
We note that the isomorphism types of the proper members of the lower central series of a $p$-group are invariants of its isoclinism classes \cite[(4.30)]{Suz86}. In particular, Sandling's condition $\gamma_2(G)^p\gamma_3(G)=1$ is invariant under isoclinism and covers the classes $\Phi_2,\Phi_4,\Phi_5,\Phi_{11}, \Phi_{12}, \Phi_{13},\Phi_{15}$. The class $\Phi_1$ corresponds to the abelian groups and was already covered by \cite{Des56}.
The class $\Phi_{14}$ is covered by \cite{BdR20}, while the class $\Phi_{35}$ is covered by \cite{BC88}. To the best of our knowledge, no other class is fully covered by existing results. Our Theorem \ref{th:MaxAbelianCentDerSubg} yields the new classes $\Phi_3$ and $\Phi_{16}$, while Corollary \ref{cor:p6} yields the new classes $\Phi_{31},\Phi_{32}, \Phi_{33}$. 
\end{remark}

\begin{corollary}\label{cor:p7}
Assume that $|G|=p^7$, $\gamma_2(G)^p\gamma_4(G)=1$, and $|G:\ZG_2(G)|=|G:\Phi(G)|=p^3$. The isomorphism type of $G$ is determined by the normalized units of the small group algebra of $kG$ if one of the following holds:
\begin{enumerate}[label=$(\arabic*)$]
\item $|\gamma_3(G)|=p$,
\item if $G$ has class $3$, then $[K_G,\gamma_2(G)]=\gamma_3(G)$. 
\end{enumerate}
\end{corollary}

\begin{proof}
If $p=2$, then we are done thanks to work of Wursthorn \cite{BKRW99}; assume thus that $p$ is odd. 
If $\gamma_3(G) = 1$, then Section \ref{sec:invariants}(6) yields the claim. Assume that $\gamma_3(G)\neq 1$. If $|\gamma_3(G)| = p$, then we are done by Corollary~\ref{cor:gamma3p}. To conclude, assume that $[K_G,\gamma_2(G)] = \gamma_3(G)$; then $[K_G,G,G] \subseteq [K_G,\gamma_2(G)]$ and so we are done by Theorem~\ref{th:K_G}.
\end{proof}

\noindent
The groups of order $p^7$ are classified, cf.\ \cite{JNOB90, OVL05}, however to the best of our knowledge there is no documentation available that allows us to give a precise count in this case of groups covered by our work as we do in Remark \ref{rmk:isoclinism}. Nonetheless, if for a fixed class $c$, we expect all primes larger than $c$ to behave ``in a somewhat analogous way for what concerns group counting'', cf.\ \cite[\S~1]{OVL05}, then we can hope to learn something about the general scenario for groups of order $p^7$ from the groups of order $7^7$. The following remark collects computational evidence for $p=5,7$ in relation to this section's results.  

\begin{remark}[(Computational evidence for $p=5,7$.)]
Denote by $f(p,n)$ the number of isomorphism classes of groups of order $p^n$ and by $\tilde{f}(p,n)$ the number of such classes in which a representative has class $3$ and derived subgroup of exponent $p$. 
Let, moreover, $g(p,n)$ denote the number of isomorphism classes of groups of order $p^n$ which satisfy the assumptions of Theorems \ref{th:MaxAbelianCentDerSubg} or \ref{th:K_G}. Set $$\mathrm{perc}(p,n)=(\tilde{f}(p,n)/f(p,n),g(p,n)/\tilde{f}(p,n),g(p,n)/f(p,n)).$$ 
The following table collects infomation on these numbers for $p=5,7$ and $n=6,7$.
\begin{center}
\begin{tabular}{|c|c|c|c|c|c|c|c|c|} 
\hline
$p$ & $f(p,6)$ & $\tilde{f}(p,6)$ & $g(p,6)$ & $\mathrm{perc}(p,6)$ & $f(p,7)$ & $\tilde{f}(p,7)$ & $g(p,7)$ & $\mathrm{perc}(p,7)$ \\ 
\hline
$5$ & $684$ & $373$ & $68$ & $(0.545,0.182,0.099)$ & $34297$  & $22436$ & $13459$ & $(0.654,0.6,0.392)$ \\ 
\hline
$7$ & $860$ & $493$ & $72$ & $(0.573,0.146,0.084)$ & $113147$ & $75948$ & $55429$ & $(0.671,0.73,0.49)$ \\ 
\hline
\end{tabular}
\end{center}
Assume $p=5$: we expand on the content of the last table.
%We list here the number of groups of order $p^6$ and $p^7$ that we are able to deal with using the results from this section. 
Up to isomorphism, there are 684 groups of order $p^6$, of which 373 have class $3$ and elementary abelian derived subgroup. Under the last assumptions, there are $50$ groups $G$ with $\Cyc_G(\gamma_2(G))$ maximal and abelian (and so covered by Theorem \ref{th:MaxAbelianCentDerSubg}) and 18 groups with second centre of index $p^3$ (covered by Corollary \ref{cor:p6}).
On the other hand, there are 34297 groups of order $p^7$, of which  22436 have class $3$ and elementary abelian derived subgroup. Within this last class, 170 groups are covered by Theorem \ref{th:MaxAbelianCentDerSubg} and 13289 groups are given by Corollary \ref{cor:p7}, contributing to a total of 13459 groups. The interpretation for $p=7$ is analogous.
\end{remark}

\noindent
The last result of this section is the following example of groups of order $5^6$ with isomorphic small group algebras. Example \ref{ex:553vs554} demonstrates that it is not possible to apply the same strategy adopted by Salim and Sandling for the groups of order $p^5$ \cite{SS95, SS96p5, SS96MC} to the groups of order $p^6$, though some of the techniques  from this section resemble theirs; see for Example \cite[Lemma 6.1]{Sal93}. We recall indeed that the work of Salim and Sandling builds upon three main subcases: groups that are determined by known group theoretical invariants, groups of maximal class \cite{SS96MC}, and the remaining groups, having class at most $3$ and elementary abelian derived subgroup, are shown to be determined by the  unit group of the small group algebra, cf.\ \cite{SS96p5} and \cite[Chapter~6]{Sal93}.

%Add references to known results. Structure of Salim's thesis: more generally the structure of the proof for the groups of order $p^5$. Separation in cases:
%\begin{enumerate}
%\item groups that are determined bby "easy" group theoretical invariants
%\item groups of maximal class \cite{SS96MC}
%\item the rest of the groups satisfy $\gamma_2^p\gamma_4=1$ and for them they show that they are determined by the unit group of the small group algebra \cite{SS96p5} \cite[Chapter~6]{SalimPHD}
%\end{enumerate}
%The way they use the small group algebra in the $p^5$-case is very similar to our style, see for example Lemma 6.1 in Salim's thesis. 
%Note:  Example \ref{ex:553vs554} shows that this strategy is not applicable in the case of groups of order $p^6$. 

\begin{example}\label{ex:553vs554}
Assume that $G = \SG(5^6, 553)$ and let $H =\SG(5^6, 554)$. It can be checked using GAP that $G$ and $H$ share all known group theoretical invariants. We will show that the unit groups of the small group algebras of $kG$ and $kH$ are isomorphic. We will do so by displaying a normal subgroup of $S=V(kG/I(kG)I(k\gamma_2(G)))$ satisfying the same relations as $H$. It will be clear from the presentations of $G$ and $H$ that both groups are $4$-generated of class $3$ and satisfy $\gamma_2(G)^p\gamma_4(G)=\gamma_2(H)^p\gamma_4(H)=1$; it thus makes sense to work at the level of small group algebras.

The SmallGroup Library in GAP provides the following polycyclic presentations of the two groups in question, where trivial commutators are omitted:
\begin{align*}
G = \langle g_1, g_2, g_3, g_4, g_5, g_6 \mid\ & g_1^5 = g_2^{25} = g_3^{25} = g_4^5 = g_5^5 = g_6^5= 1, \\ & [g_2,g_1] = g_6, [g_4, g_3] = g_5, [g_5, g_4] = g_6, g_2^5 = g_3^5 = g_6^{-1}\rangle;\\
H = \langle h_1, h_2, h_3, h_4, h_5,h_6 \mid\ & h_1^5 = h_2^{25} = h_3^{25} = h_4^5 = h_5^5 = h_6^5 = 1, \\ & [h_2,h_1] = h_6, [h_4, h_3] = h_5, [h_5, h_4] = h_6, h_2^5 = h_3^5 = h_6^{-2}\rangle.
\end{align*}
Since our methods are more suitable for controlling differences between commutators rather than between $p$-th powers, we translate the last presentations to
\begin{align}\label{eq:56554}
\nonumber
G = \langle g_1, g_2, g_3, g_4, g_5, g_6 \mid\ & g_1^5 = g_2^{25} = g_3^{25} = g_4^5 = g_5^5 = g_6^5 = 1,\\ \nonumber & [g_2,g_1] = g_6^{-1}, [g_4, g_3] = g_5, [g_5, g_4] = g_6^{-1}, g_2^5 = g_3^5 = g_6 \rangle; \\ 
H = \langle h_1, h_2, h_3, h_4, h_5, h_6 \mid \ & h_1^5 = h_2^{25} = h_3^{25} = h_4^5 = h_5^5 = h_6^5 = 1,\\ \nonumber & [h_2,h_1] = h_6^2, [h_4, h_3] = h_5, [h_5, h_4] = h_6^2, h_2^5 = h_3^5 = h_6\rangle. 
\end{align}
We note that $A=A(G|g_1,\ldots,g_4)/(A\cap Z(S))$ is generated by the elements $a = 1 + \bar{g_3}\bar{g_4}$ and $b = 1 + \bar{g_4}^2$. Moreover, one can compute that $\Cyc_S(b) = \langle A, \gamma_2(G),  g_1, g_2, g_4 \rangle$ and $[g_3,b] = [g_3,g_4,g_4] = [g_5^{-1},g_4] = g_6$. Define now
\[\tilde{h}_1 = g_1^{-2}b^{-2}, \quad \tilde{h}_2 = g_2, \quad \tilde{h}_3 = g_2^{-2} g_3^{-2}, \quad \tilde{h}_4 = g_4 \]
and set  $\tilde{H} = \langle \tilde{h}_1, \tilde{h}_2, \tilde{h}_3, \tilde{h}_4 \rangle$. We claim that $\tilde{H} \cap A = 1$ and $\tilde{H} \cong H$.
To this end, define 
\[\tilde{h}_5 = [\tilde{h}_4,\tilde{h}_3] = [g_4,g_2^{-2}g_3^{-2}] = [g_4,g_3^{-2}] = g_5^{-2} \text{ and } \tilde{h}_6 =\tilde{h}_2^5 = g_6. \]
We will show that the relations of $\tilde{H}$ with respect to the  generators $\tilde{h} _1,\ldots,\tilde{h}_6$ are exactly those from $H$'s presentation in \eqref{eq:56554}.
It is not difficult to show that
\[\tilde{h}_1^5 =\tilde{h}_2^{25} =\tilde{h}_3^{25} =\tilde{h}_4^5 = 1 \text{ and } \tilde{h}_3^5 = (g_2^{-2}g_3^{-2})^5 = g_6^{-4} = g_6 = \tilde{h}_6,\]
settling the $p$-th powers. Moreover, we compute
\begin{align*}
[\tilde{h}_2, \tilde{h}_1] &= [g_2, g_1^{-2}b^{-2}] = [g_2, g_1]^{-2} = g_6^2 = \tilde{h}_6^2, \\ 
[\tilde{h}_3, \tilde{h}_1] &= [g_2^{-2}g_3^{-2}, g_1^{-2} b^{-2}] = [g_2^{-2},g_1^{-2}][g_3^{-2}, b^{-2}] = g_6 g_6^{-1} = 1, \\
[\tilde{h}_4, \tilde{h}_1] &= [g_4, g_1^{-2}b^{-2}] = 1,  \
[\tilde{h}_3, \tilde{h}_2] = [g_2^{-2}g_3^{-2}, g_2] = 1, \
 [\tilde{h}_4, \tilde{h}_2] = [g_4,g_2] = 1, \\
 [\tilde{h}_5, \tilde{h}_1] &= [\tilde{h}_5, \tilde{h}_2] = [\tilde{h}_5, \tilde{h}_3] = 1, \
[\tilde{h}_5, \tilde{h}_4] = [g_5^{-2}, g_4] = [g_5,g_4]^{-2} = g_6^2 = \tilde{h}_6^2. 
\end{align*} 
Since $\tilde{h}_6$ is clearly central in $\tilde{H}$, all the relations from \eqref{eq:56554} hold. This shows that $\tilde{H}$ is isomorphic to $H$, implying that the unit groups of the small group algebras of $kG$ and $kH$ are isomorphic. Having a closer look on the elements of $\tilde{H}$ we see that they are also linearly independent in $kG/I(kG)I(k\gamma_2(G))$, as so are the elements of $G$. So $\tilde{H}$ is a group basis of $kG/I(kG)I(k\gamma_2(G))$. Hence the small group algebras of $G$ and $H$ are also isomorphic. 
\end{example}

\noindent
We close the present section with an account of the history of pairs of groups that are not told apart by the small group algebra. 
To the best of our knowledge, the first example appearing in the literature dates back to \cite[Section 4]{Bag99} and concerns groups of order at least $p^5$, of maximal class with an elementary abelian maximal subgroup (in particular the forth element of the lower central series is, in this case, non-trivial). In \cite{HS07}, the authors display groups of order $16$ and $32$ with the same relative small group algebra: for all such groups the derived subgroup is cyclic of order $4$, so the considered small group algebra is different from the one from this paper. We also wish to mention two examples that are mentioned in the literature, but were never published. The first one, attributed to Kimmerle and Scott, figures in \cite[Footnote 6.e after 6.25 from Chapter 6]{San84} and \cite{HS07}. The second one, attributed to Wursthorn, 
concerns four groups of order $5^5$ and maximal class that have all isomorphic small group algebras and is referred to in
 \cite[page 1073]{SS96MC}.

\section{Lower central series invariants}\label{sec:2genClass3}
In \cite{BdR20}, the (MIP) is solved for $2$-generated $p$-groups of class $2$, while in \cite[Proposition 2.5]{MM20} it is solved for $2$-generated groups of class $3$ in which the derived subgroup is elementary abelian. In this section, we consider $2$-generated groups of class $3$, but drop the assumption that the derived subgroup is elementary abelian.
The main result of this section is the following.

\begin{theorem}\label{th:2gen-class3-orderG3}
Assume that $p$ is odd and $H$ is a group such that $kG \cong kH$. Assume moreover that $G$ is $2$-generated with $\gamma_3(G)$ central of exponent $p$. Then $H$ has the same class as $G$ and, for each integer $i\geq 2$, one has $\gamma_i(G) \cong \gamma_i(H)$.
\end{theorem}

\noindent
Under the assumptions of Theorem \ref{th:2gen-class3-orderG3}, proving that $\gamma_3(G) \cong \gamma_3(H)$ turns out to be the most challenging task. Indeed, the number $|\gamma_3(G)|$ is not a known group theoretical invariant and cannot be derived from known group theoretical invariants either, as the following example shows.

\begin{example}
For $G = \SG(3^7,19)$ and $H = \SG(3^7,43)$, a quick check using the software of \cite{MM20} shows that $G$ and $H$ share all known group theoretical invariants, but do not satisfy $|\gamma_3(G)|=|\gamma_3(H)|$.
\end{example}

\begin{lemma}\label{lemma:structure1}
Assume that $G$ has class $3$, $\gamma_2(G)/\gamma_3(G)$ is cyclic, and $\gamma_3(G)$ is elementary abelian. Then $\Gamma(G)=\gamma_2(G)^p\gamma_3(G)$ and $\Gamma(G)$ is maximal in $\gamma_2(G)$.
\end{lemma}

\begin{proof}
The subgroup $\gamma_3(G)$ is contained in $\Gamma(G)$ and the commutator map $G\times \gamma_2(G)\rightarrow \gamma_3(G)$ is bilinear.  It follows that $\gamma_2(G)^p$ is contained in $\Gamma(G)$ and so, $\gamma_2(G)/\gamma_3(G)$ being cyclic, we get that $\Gamma(G)=\gamma_2(G)^p\gamma_3(G)$ has index $p$ in $G$.
\end{proof}

\begin{proposition}\label{prop:invs}
Assume that $G$ has class $3$, $\gamma_2(G)/\gamma_3(G)$ is cyclic, and $\gamma_3(G)$ is elementary abelian. Let, moreover, $H$ be a group such that $kG\cong kH$. Then $H$ has class $3$, the quotient $\gamma_2(H)/\gamma_3(H)$ is cyclic, $\gamma_2(G) \cong \gamma_2(H)$, and $\gamma_3(H)$ is elementary abelian. 
\end{proposition}

\begin{proof}
Thanks to the list in Section~\ref{sec:invariants} we know that
\begin{enumerate}[label=$(\arabic*)$]
\item $G/\gamma_2(G)\cong H/\gamma_2(H)$,
\item $\Gamma(G)\cong \Gamma(H)$,
\item for each $n$: we have $D_n(\gamma_2(G))/D_{n+1}(\gamma_2(G))\cong D_n(\gamma_2(H))/D_{n+1}(\gamma_2(H))$.
\end{enumerate}
From $(1)$ and $(2)$ combined with Lemma \ref{lemma:structure1}, we derive that $\Gamma(H)$ is maximal in $\gamma_2(H)$ and, in particular, $H$ has class at least $2$ and $\gamma_2(H)$ is abelian. We next show that $\gamma_3(H)$ is contained in $\Gamma(H)$; for a contradiction, we assume that this is not the case. It follows that $\gamma_2(H)=\Gamma(H)\gamma_3(H)$ and so, $\Gamma(H)$ being central in $H$, we get that 
\[
\gamma_3(H)=[H,\gamma_2(H)]=[H,\Gamma(H)\gamma_3(H)]=[H,\gamma_3(H)]=\gamma_4(H).
\]
In particular, $H$ has class $2$, contradicting  Section~\ref{sec:invariants}(8). This proves that $H$ has class $3$.
The fact that the abelian groups $\gamma_2(G)$ and $\gamma_2(H)$ are isomorphic follows from (3). 
 
We now claim that $\Gamma(H)=\gamma_2(H)^p\gamma_3(H)$. For this, assume first that $\gamma_2(H)$ does not have exponent $p$. We remark that, by Lemma \ref{lemma:structure1}, we have that $\Gamma(G)=\gamma_2(G)^p\gamma_3(G)$ and so, $\gamma_2(G)/\gamma_3(G)$ being cyclic, we also have that $\Gamma(G)=\mu_{\exp(\gamma_2(G))/p}(\gamma_2(G))$. Given the isomorphism type of $\gamma_2(G)$, this is the only subgroup of $\gamma_2(G)$ that is isomorphic to $\Gamma(G)$%\ChMima{(to see this you can observe that $\Gamma(G)$ is the unique subgroup of $\gamma_2(G)$ containing $\mu_p(\gamma_2(G))$ and such that $\Gamma(G)/\mu_p(\gamma_2(G))$ is the maximal subgroup of the cyclic subgroup $\gamma_2(G)/\mu_p(\gamma_2(G))$)}
. It follows from (2) that 
$$\Gamma(H)=\mu_{\exp(\gamma_2(H))/p}(\gamma_2(H))=\gamma_2(H)^p\gamma_3(H)$$ and thus that $\gamma_2(H)/\gamma_3(H)$ is cyclic. Now, $\gamma_2(H)^p$ and $\gamma_3(H)$ being central, we have that
\[
\gamma_3(H)^p=[H,\gamma_2(H)]^p=[H,\gamma_2(H)^p] = 1,
\]
equivalently $\gamma_3(H)$ has exponent $p$.

Assume now that $\gamma_2(G)^p = 1$. Then $\gamma_2(G)$ is elementary abelian and so is $\gamma_2(H)$. In particular, $\gamma_3(H)$ is elementary abelian. Now, relying on Section~\ref{sec:invariants}(6), the (Sandling) quotients $G/\gamma_3(G)$ and $H/\gamma_3(H)$ are isomorphic and so $|\gamma_3(G)|=|\gamma_3(H)|$. Since $\gamma_2(G)/\gamma_3(G)$ is cyclic, the equalities $$|\gamma_2(G):\gamma_3(G)|=p=|\gamma_2(H):\gamma_3(H)|$$ hold, proving that $\gamma_2(H)/\gamma_3(H)$ is cyclic.
\end{proof} 

%\begin{lemma}\label{lem:classsum}
%Let $X$ be a $p$-group and $x \in X$ such that $x^X = x \langle y \rangle$ for some $y \in X$. Assume the order of $y$ is $n$ and $y \neq 1$. Then the class sum of $x$ in $kX$ equals
%$(1+\bar{x}) \bar{y}^{n-1}. $
%\end{lemma}
%
%\begin{proof}
%Noting that $n \equiv 0 \mod p$ and we have that the class sum of $x$ equals
%$$\sum_{i=1}^n x y^i = \sum_{i=1}^n \overline{x y^i} = \sum_{i=1}^n \bar{x} + \sum_{i=1}^{n-1} \overline{y^i} + \bar{x}\sum_{i=1}^{n-1} \overline{y^i} = (1+\bar{x})\sum_{i=1}^{n-1} \overline{y^i}. $$
%Now we use the formula $\sum_{m=r}^\ell \binom{m}{r} = \binom{\ell+1}{r+1}$, which holds for all integers $r,\ell,m$, and that $p$ divides $\binom{n}{r}$ for any $r < n$ as $n$ is a power of $p$ to obtain  
%$$ \sum_{i=1}^{n-1} \overline{y^i} = \sum_{i=1}^{n-1} \sum_{j=1}^i \binom{i}{j} \overline{y}^j = \sum_{j=1}^{n-1} \sum_{i=j}^{n-1} \binom{i}{j} \overline{y}^j = \sum_{j=1}^{n-1} \binom{n}{j+1} \overline{y}^j = \binom{n}{n}\overline{y}^{n-1} = \overline{y}^{n-1}. $$
%\end{proof}

\begin{lemma}\label{lemma:J_X}
Let $X$ be a $2$-generated $p$-group of class $3$ such that $\gamma_3(X)$ is elementary abelian. Let moreover $J_X$ be a subgroup of $X$ containing $T=\Phi(\gamma_2(X))\cap \gamma_3(X)$ such that 
$J_X/T=\Cyc_{X/T}(\gamma_2(X)/T).$
Then the following hold: 
\begin{enumerate}[label=$(\arabic*)$]
\item $p\leq |\gamma_3(X)|\leq p^2$,
\item $|X:J_X|=p$,
\item if $|\gamma_3(X)|=p^2$, then $\Phi(X)\neq X^p\gamma_3(X)$.
\end{enumerate}
\end{lemma}

\begin{proof}
Since $X$ is $2$-generated, the quotient $\gamma_2(X)/\gamma_3(X)$ is cyclic and thus the commutator map induces a surjective homomorphism $X/\Phi(X)\otimes \gamma_2(X)/\Gamma(X)\rightarrow \gamma_3(X)$. As a consequence of Lemma \ref{lemma:structure1}, the index $|\gamma_2(X):\Gamma(X)|$ is $p$ and thus $|\gamma_3(X)|\leq|X/\Phi(X)|=p^2$. Since $X$ is assumed to have class $3$, we also have $|\gamma_3(X)|\geq p$. Let $\varphi: X  \rightarrow X/T$ be the natural projection. It follows from (1) and the fact that $\gamma_2(X)/\gamma_3(X)$ is cyclic that $\gamma_3(\varphi(X))$ has order $p$. Now, the commutator map induces a non-degenerate bilinear map
\[
\varphi(X)/\varphi(J_X)\times\gamma_2(\varphi(X))/\Gamma(\varphi(X))\rightarrow\gamma_3(\varphi(X))
\]
and so, $\gamma_3(\varphi(X))$ having order $p$, we derive $|X:J_X|=|\gamma_2(X):\Gamma(X)|=p$. To conclude, assume that $|\gamma_3(X)|=p^2$ and assume, for a contradiction, that $\Phi(X)=X^p\gamma_3(X)$. Then the commutator map 
induces a bilinear map 
\[
X/\Phi(X)\times X^p\gamma_3(X)/\Gamma(X)\rightarrow\gamma_3(X)
\]
whose image generates $\gamma_3(X)$. From the fact that $\gamma_3(X)$ is elementary abelian, we  get that the last bilinear map has a non-trivial right kernel and thus $|\gamma_3(X)|$ has order at most $p^{1\cdot\log_p|X^p\gamma_3(X):\Gamma(X)|}=p$. Contradiction. 
\end{proof}

\noindent
The remaining part of this section will be devoted to the proof of Theorem \ref{th:2gen-class3-orderG3}. To this end, the following \underline{assumptions} will hold until the end of Section \ref{sec:2genClass3}. Assume $p$ is odd and let $H$ be a $p$-group such that $kG \cong kH$. Assume, additionally, that $G$ is $2$-generated with $\gamma_3(G)$ central of exponent $p$. Without loss of generality, we assume that $kG=kH$ and so we view $G$ and $H$ as group bases of the same modular group algebra. If the class of $G$ is smaller than $3$, then it follows from \cite{BdR20} that $G \cong H$ and in particular Theorem \ref{th:2gen-class3-orderG3} holds. We assume that $G$ has class $3$.

 From group theoretical invariants and Proposition \ref{prop:invs},  we know that $H$ is $2$-generated of class $3$, $\gamma_2(G) \cong \gamma_2(H)$, and that $\gamma_3(H)$ is elementary-abelian. We will show that $|\gamma_3(G)| = |\gamma_3(H)|$. For a contradiction, we \underline{assume} that $|\gamma_3(G)| < |\gamma_3(H)|$. 

We fix the notation. Let $g_G,h_G$ and $g_H,h_H$ be generators of $G$ and $H$, respectively. In case it is clear which group we are considering, we will omit indices. As a consequence of Lemma \ref{lemma:J_X}(2), we assume without loss of generality that $h_G$ and $h_H$ are elements of $J_G$ and $J_H$, respectively. 
 Set $c = [g,h]$ and let $p^{t+1}$ denote the order of $c$. Observe that, thanks to Lemma \ref{lemma:J_X}(1), we have that $|\gamma_3(G)| = p$ and $|\gamma_3(H)| = p^2$. %In particular, we also have, up to possible rescaling (namely up to possibly replacing $g_H$ with a power), that
Summarizing, in both groups we have
\[
c = [g,h],  \quad  |c|=p^{t+1}, \quad  [c,g]=d \neq 1, \quad d^p=1, 
\]
while, up to possible rescaling (namely up to possibly replacing $g_H$ with a power), the following hold independently
\begin{alignat*}{3}
[c_G,h_G] & =1, \quad  \ \
\gamma_3(G)&&=\gen{[c_G,g_G]}=\gen{d_G},   \\
[c_H,h_H]&=c_H^{p^t},  \quad \gamma_3(H)&&=\gen{[c_H,h_H],[c_H,g_H]}=\gen{c_H^{p^t}}\times\gen{d_H}.
 \end{alignat*}  
From the fact that $\gamma_2(G)$ and $\gamma_2(H)$ are isomorphic, it follows 
that
$$\gamma_2(G)=\langle c_G \rangle \times \langle d_G \rangle \textup{ and } \gamma_2(H) = \langle c_H \rangle \times \langle d_H \rangle.$$
 Moreover, from Lemma~\ref{lemma:comm-formulas} we derive that
 $
[G^{p^{t+1}},G]=[H^{p^{t+1}},H]=1$
while, for any integer $i$, the following hold:
\[[h^i g^{p^t},h] \neq 1 \textup{ and }[g,h^{p^t}g^i] \neq 1.
 \]
Let $G^{\ab}=G/\gamma_2(G)$ denote the abelianization of $G$, as in Section \ref{sec:preliminaries-small}, and let
 $$\pi:G \longrightarrow G/G^{p^{t+1}}\gamma_2(G)\cong G^{\ab}/(G^{\ab})^{p^{t+1}}
$$ denote the natural projection. Then $\pi(g_G)$ and $\pi(h_G)$ both have order $p^{t+1}$ and 
$\pi(G)=\gen{\pi(g_G)}\oplus\gen{\pi(h_G)}.$
It follows, for every $\ell \leq t+1$, that 
$$(G^{p^\ell} \setminus G^{p^{t+1}}) \cap \gamma_2(G) \subseteq \langle c_G^{p^\ell} \rangle$$ 
and an analog statement can be derived for $H$ by looking at $H/H^{p^{t+1}}\gamma_2(H)$.
Note that, if $t=0$, then Section~\ref{sec:invariants}(6) yields that $G \cong H$ and so Theorem \ref{th:2gen-class3-orderG3} follows. We \underline{assume} that $t \geq 1$. 
We fix, moreover, $r$ to be an integer such that, for any other integer $r_0$, one has  
$$\wt\left( c_G^{r} d_G \right) \geq \wt\left( c_G^{r_0} d_G \right).$$ 

\begin{lemma}\label{lemma:JenningsBasisChoice}
The following hold:
\begin{enumerate}[label=$(\arabic*)$]
\item $\wt(c_G)=\wt(c_H) = 2$, 
\item for each $\ell \leq t$, one has $\wt\left(g_G^{p^\ell}\right) = \wt\left(h_G^{p^\ell}\right) =\wt\left(g_H^{p^\ell}\right) = \wt\left(h_H^{p^\ell}\right)= p^\ell$, 
\item for $X\in\graffe{G,H}$, there exists a Jennings tuple $\cor{J}=(g_1,\ldots,g_m)$ of $X$ satisfying
\[\graffe{g_1,\ldots,g_m}\supseteq \graffe{g,g^p,\ldots,g^{p^t}, h,h^p,\ldots,h^{p^t}, c, c^p,\ldots,c^{p^t}}\cup\begin{cases}
\emptyset & \textup{if } X=H, \\
\graffe{c^rd} & \textup{if } X=G,  
\end{cases}
\]
and, if $i,j\in\graffe{1,\ldots,m}$ and $\ell\leq t$ are such that $g_i=h^{p^\ell}$ and $g_j=g^{p^\ell}$, then $i<j$.
\end{enumerate}
\end{lemma}

\begin{proof}
To lighten the notation, we will avoid the use of indices in this proof.
 
(1) Thanks to Lemma~\ref{lemma:J_X}(3), the element $c$ is not a product of $p$-th powers and has thus weight $2$. 
%Thanks to Lemma \ref{lemma:J_X}(3), the element $c$ does not belong to $\langle g \rangle$ or $\langle h \rangle$ and, in particular, $\bar{c}$ has weight $2$ in $kG$. 

(2) As $g^{p^t}$ and $h^{p^t}$ are not elements of the derived subgroup, we have, for all $\ell \leq t$, that  $$\wt\left(g^{p^\ell}\right) = \wt\left(h^{p^\ell}\right) = p^\ell.$$

(3) We will prove the existence of $\cor{J}$ for $X=G$; similar arguments apply to $H$.
We define 
$$\tau = \{g,g^p,\ldots,g^{p^t}, h,h^p,\ldots,h^{p^t},c, c^p,\ldots,c^{p^t}\}.$$
 To show that the elements of $\tau$ figure as entries of a Jennings tuple of $G$, we shall show that elements of $\tau$ with the same weight $i$ in $G$ are linearly independent modulo $D_{i+1}(G)$. For non-negative integers $i,j<t$, 
it follows from the properties of the dimension subgroups that $\wt\left(g^{p^i}\right)=\wt\left(g^{p^j}\right)$ if and only if $i=j$; analogously it holds that $\wt\left(h^{p^i}\right)=\wt\left(h^{p^j}\right)$ or $\wt\left(c^{p^i}\right)=\wt\left(c^{p^j}\right)$ if and only if $i=j$.

We proceed by showing that, for any $\ell \leq t$, the weight of $c^{p^\ell}$ is different from the weight of any other element in $\tau$. If $\wt\left(c^{p^\ell} \right) = 2p^\ell$, the same argument as above yields the claim.
Assume now that,  for some $\ell \leq t$, one has $\wt\left(c^{p^\ell}\right) \neq 2p^\ell$, equivalently that $c^{p^\ell} \in D_{2p^\ell +1}(G) = G^{p^{\ell +1}}$.
From the inclusion $(G^{p^{\ell+1}} \setminus G^{p^{t+1}}) \cap \gamma_2(G) \subseteq \langle c^{p^{\ell+1}} \rangle$, we conclude that $c^{p^\ell} \in G^{p^{t+1}}$ and so $c^{p^\ell}$ has weight larger than that of any element in the set $\{g,g^p,\ldots,g^{p^t},h,h^p,\ldots,h^{p^t} \}$. 

We have proven that the only elements in $\tau$ having the same weight come in pairs $(g^{p^\ell}, h^{p^\ell})$, for $\ell \leq t$. Moreover, two elements $g^{p^\ell}$ and $h^{p^\ell}$ are linearly independent in $D_{p^\ell}(G)/D_{p^\ell+1}(G)$ as they span the $2$-dimensional $k$-vector space $(G^{\ab})^{p^\ell}/(G^{\ab})^{p^{\ell+1}}$. 
In the choice of the Jennings tuple $\cor{J}$, we take without loss of generality $h^{p^\ell}$ to precede $g^{p^\ell}$.

We finally assume, for a contradiction, that $c^rd$ has the same weight as some other element of $\tau$. Note that, as $d \in \gamma_3(G)$, we have $\wt\left(c^rd \right) \geq 3$. If $\wt\left( c^rd \right) = 3$, we are reduced to considering the case $p=3$ and showing that $g^3$, $h^3$ and $c^rd$ are linearly independent in $D_3(G)/D_4(G) = G^3\gamma_2(G)^3\gamma_3(G)/G^9\gamma_2(G)^3$. Now, if $g^3$, $h^3$ and $c^rd$ are linearly dependent, then there exist integers $0 \leq i_1,i_2,i_3 \leq 2$, of which at least two are non-zero modulo $3$, such that $(h^3)^{i_1}(g^3)^{i_2} (c^rd)^{i_3} \in G^9\gamma_2(G)^3$. It then follows that $\pi((h^3)^{i_1}(g^3)^{i_2}) \in (G^{\ab})^9$, contradicting the structure of $\pi(G)$. 

Assume now, for some integer $s$, that $\wt\left( c^rd \right) = p^s > 3$. It follows that $c^rd \in G^{p^s} \setminus G^{p^{s+1}}$. If $s \leq t$, then the inclusion $(G^{p^{s+1}} \setminus G^{p^{t+1}}) \cap \gamma_2(G) \subseteq \langle c^{p^{s+1}} \rangle$ provides a contradiction. So $s \geq t+1$ and, in particular, $c^rd$ has weight different from the weights of any $g^{p^\ell}$ and $h^{p^\ell}$ with $\ell \leq t$.  In conclusion, if, for some $\ell \leq t$, the elements $\overline{c^rd}$ and $\overline{c^{p^\ell}}$ are linearly dependent modulo $D_{\wt\left( c^rd \right) + 1}(G)$  then there is an integer $m$ such that $\wt\left(c^rd(c^{p^\ell})^m\right) > \wt\left(c^rd\right)$, contradicting the choice of $r$. 
\end{proof}

\noindent
We fix Jennings tuples of $G$ and $H$ and \underline{assume} them to satisfy the properties listed in Lemma~\ref{lemma:JenningsBasisChoice}.
Let, moreover, $m$ be an integer such that, for any integer $m_0$, one has 
$$\wt\left( c_G^{p^t}d_G^m \right) \geq \wt\left( c_G^{p^t}d_G^{m_0} \right).$$
Set $w = \wt\left( c_G^{p^t}d_G^m \right)$ and observe that
$$\wt\left(c_G\right) > \wt\left(c_G^p\right) > \ldots > \wt\left(c_G^{p^t}\right) \geq \wt\left(c^{p^t}d^m \right) = w.  $$
In particular, no element among $c_G,c_G^p,\ldots,c_G^{p^t}$ has weight bigger than $w$.
\vspace{8pt}\\
\noindent
Until the end of this section, \underline{let}  $\mathcal{K}$ be the smallest ideal of $kG$ containing $\overline{d_G}$ and $I(kG)^{2w}$. 
We define 
$$Z=\ZG(kG) + \mathcal{K}/\mathcal{K} \textup{ and } \mu_p^0(Z) = \{a \in Z \mid a^p = 0 \}.$$
By abuse of notation, we will denote the elements of $Z$ as they were elements of $kG$.
Then $\mu_p^0(Z)$ is a $k$-vector space containing $\overline{c_G^{p^t}}$, $\overline{c_H^{p^t}}$ and $\overline{d_H}$ in which, by definition of $\cor{K}$, the image of $\overline{c_G^{p^t}}$ is non-trivial (see also \eqref{eq:c^2} below).  We fix a subspace $L$ of $\mu_p^0(Z)$ such that 
$$\mu_p^0(Z) = \langle \overline{c_G^{p^t}} \rangle \oplus L$$
 and let $\mathcal{L}$ be the ideal of $kG$ containing $\mathcal{K}$ such that $\mathcal{L}/\mathcal{K} = L$. There then exist $\alpha$ and $\beta$ in $k$ such that
$$ \overline{d_H} \equiv \alpha \overline{c_G^{p^t}} \bmod \mathcal{L} \ \text{ and } \ \overline{c_H^{p^t}} \equiv \beta \overline{c_G^{p^t}} \bmod \mathcal{L}.$$
Moreover, modulo $\cor{L}$, the following equalities hold:
\begin{equation}\label{eq:c^2}
0 \not\equiv \overline{c_G^{p^t} d^m} = \overline{c_G^{p^t}} + \overline{d_G^m} + \overline{c_G^{p^t}} \overline{d_G^m} \equiv \overline{c_G^{p^t}} \bmod \mathcal{L}
 \ \textup{ and therefore } \
0 \equiv \overline{c_G^{p^t} d^m}^2 \equiv \overline{c_G^{p^t}}^2 \bmod \mathcal{L}.
\end{equation}
We let $\varphi_L:kG\rightarrow kG/\cor{L}$ be the canonical projection. 

\begin{lemma}\label{lemma:beta}
One has $\beta \neq 0$.
\end{lemma}

\begin{proof}
Write $|G| = p^m$ and denote by $\mathcal{J} = (g_1,\ldots,g_m)$ the chosen Jennings tuple of $G$. 
As remarked in Section \ref{sec:Ideals}, any element of $G$ can be written as $\prod_{i=1}^m g_i^{\alpha_i}$ with  $0 \leq \alpha_i \leq p-1$.
 Let $\tilde{w}$ be the weight of $\overline{c_H^{p^t}}$ and note that, by the theory of Zassenhaus ideals, cf. Section~\ref{sec:Ideals}, we can write
$$\overline{c_H^{p^t}} = \bar{a} + y \text{ where } \ a \in D_{\tilde{w}}(G), \ y \in I(kG)^{\tilde{w} + 1}. $$
Let $0 \leq \alpha_1,\ldots, \alpha_m \leq p-1$ be such that $a = \prod_{i=1}^m g_i^{\alpha_i}$. Thanks to Lemma~\ref{lemma:BasicJenningBasisFormulas} we have
 \[ \bar{a} = \sum_{i=1}^m \alpha_i \overline{g_i} + r, \]
 where the Jennings basis elements in the support of $r$ are products of at least two elements from $\overline{\mathcal{J}}$ and are linearly independent from the elements in $\overline{\mathcal{J}}$. Note that the weight of $\bar{a}$ is the smallest weight among the weights of $\overline{g_i}$'s satisfying $\alpha_i \neq 0$. As the weight of $\bar{a}$ equals $\tilde{w}$, we conclude that $\alpha_i = 0$ whenever the weight of $\overline{g_i}$ is smaller than $\tilde{w}$ and there must exist an $i$ such that $\overline{g_i}$ has weight $\tilde{w}$ and $\alpha_i \neq 0$. 
 
Observe now that $\overline{c_H^{p^t}} \in kGI(k\gamma_2(G))$, as the ideal $kGI(k\gamma_2(G))$ is canonical. It follows that,
in the expression of $\overline{c_H^{p^t}}$ as linear combination of  elements in the Jennings basis on $\cor{J}$, each support element features a factor from $\overline{\gamma_2(G)}$, cf. Section~\ref{sec:Ideals}. In particular, if the weight of $\overline{g_i}$ is $\tilde{w}$ and $\alpha_i \neq 0$, then $g_i \in \gamma_2(G)$. 

We write $\overline{c_H} \in kGI(k \gamma_2(G))$ as $\overline{c_H} = \delta_1 \overline{c_G} + \delta_2 \overline{d_G}$ for some $\delta_1, \delta_2 \in kG$. 
Since $\gamma_2(G)$ is abelian and $\overline{d_G}^p = 0$,
Lemma~\ref{lemma:BasicJenningBasisFormulas} yields elements $\varepsilon_0, \ldots, \varepsilon_{p-1} \in kG$ such that
\[  \overline{c_H^{p^t}} = \overline{c_H}^{p^t} = (\delta_1 \overline{c_G} + \delta_2 \overline{d_G})^{p^t} = \sum_{j=0}^{p-1}\varepsilon_j \overline{c_G}^{p^t-j} \overline{d_G}^j. \]
 As $d_G$ is central in $G$ and $[\overline{c_G}, kG] \subseteq kG \overline{d_G}$, the elements of $\cor{J}$ in the support of $\varepsilon_j \overline{c_G}^{p^t-j} \overline{d_G}^j$ have at least $p^t$ factors from $\{\overline{c_G}, \overline{d_G} \}$.
However, at the same time there is an element $\overline{g_i}$ in the support of $\bar{a}$ such that $g_i \in \gamma_2(G)$, i.e.\ $\overline{g_i}$ is an element from $\{\overline{c_G},...,\overline{c_G^{p^t}}, \overline{c_G^r d_G} \}$. It follows that, if $g_i = c_G^{p^t}$, then $\alpha_i \neq 0$ and in particular $\beta \neq 0$.
\end{proof}

\begin{lemma}\label{lem:Putd_HIntoL}
There exists $\delta\in k$ such that $d_H(c_H^{p^t})^\delta\equiv 1\bmod \cor{L}$.
\end{lemma}

\begin{proof}
Define $\delta=-\alpha\beta^{-1}$ and observe that $\delta$ is well-defined thanks to Lemma \ref{lemma:beta}. As a consequence of \eqref{eq:c^2}, one has $\overline{c_H^{p^t}}^2 \equiv 0 \bmod \mathcal{L}$ and so we compute
\begin{align*}
\overline{d_H(c_H)^{\delta p^t}}&\equiv\overline{d_H}+\delta\overline{c_H^{p^t}}+\delta\overline{d_H}\cdot \overline{c_H^{p^t}}
\equiv\alpha\overline{c_G^{p^t}}+\delta\beta\overline{c_G^{p^t}}+\alpha\beta\delta\overline{c_G^{p^t}}^2\equiv 0 \bmod \cor{L}.
\end{align*}
Removing bars yields the claim.
%where the last congruence follows from the definition of $\delta$ and the fact that $\overline{c_G^{p^t}}^2 \equiv 0 \bmod \mathcal{L} $.
\end{proof}

\noindent
For the remaining part of Section \ref{sec:2genClass3}, \underline{let} $\A$ denote the algebra $\A = kG/\mathcal{L} = kH/\mathcal{L}$. If $\overline{d_H} \notin \mathcal{L}$, then we replace $g_H$ by $g_Hh_H^\delta$ where $\delta$ is chosen as in Lemma~\ref{lem:Putd_HIntoL}. This choice has no influence on the results from this section, nonetheless it guarantees that $\overline{d_H} \in \mathcal{L}$. In particular, the images of $\overline{g_H}$ and $\overline{c_H}$ in $\A$ commute. Our strategy is to show that the ideal $[\A,\A]\A \cap \ZG(\A)$ has conflicting properties, when viewed as an ideal of $kG/\mathcal{L}$ or $kH/\mathcal{L}$, respectively. The choice of the ideal is inspired by Sandling's argument  \cite[Lemma 6.10]{San84} relying on the ideal $[kX,kX]kX \cap Z(kX)$ of $kX$ to show that the isomorphism type of $\Gamma(X)$ is an invariant of $kX$; cf.\ Section \ref{sec:invariants}(5). We point out that, under our assumptions, we have indeed that $\Gamma(\varphi_L(G)) \not\cong \Gamma(\varphi_L(H))$. Compared to \cite[Lemma~6.10]{San84}, however, we do not have a well-understood basis of the center of $\A$ resulting in more complicated calculations on our side. 

 \begin{lemma}\label{lem:Centralizerx}
 There exists $x\in I(kG)^3 \cap kGI(\gamma_2(kG))$ such that $ \varphi_L(\overline{c_H}+x) \in [\A,\A]\A\cap\ZG(\A)$ and the following hold:
 \[
[x,\overline{c_G}]\equiv[x,\overline{c_H}] \equiv [x,\overline{g_H}] \equiv 0 \bmod \mathcal{L}. 
 \]
 \end{lemma}
 
 \begin{proof}
The elements $\overline{c_G}$ and $\overline{c_H}$ are both elements of the canonical ideal $kGI(k\gamma_2(G))$ and have weight $2$; moreover, modulo $I(kG)^3$, they are, up to invertible scalars, unique with this property. It follows that there exist elements $x\in I(kG)^3 \cap kGI(\gamma_2(kG))$ and $\lambda\in k^*$ such that 
 \[
\varphi_L( \overline{c_H}+x) = \varphi_L(\lambda\overline{c_G}) \in [\A,\A]\A\cap \ZG(\A). 
 \] 
For such an $x$, we use the facts that $\varphi_L(\overline{c_G})$ is central and $[\varphi_L(\overline{c_H}),\varphi_L(\overline{g_H})]=0$ to derive that 
\[0 = [\varphi_L(x),\varphi_L(\overline{c_G})] = [\varphi_L(x), \varphi_L(\overline{c_H})] = [\varphi_L(x), \varphi_L(\overline{g_H})].\]
 \end{proof}

\begin{proof}[Proof of Theorem~\ref{th:2gen-class3-orderG3}]
We will show that the element $x$ introduced in Lemma~\ref{lem:Centralizerx} cannot exist, which will provide a contradiction to $kG = kH$.  

As a generating set of $\A$ as a vector space, we will use the image under $\varphi_L$ of the Jennings basis of $kG=kH$ obtained from the Jennings tuple of $H$. To lighten the notation, we will drop the index $H$ in our calculations. Moreover, we will abuse notation and denote the elements of $\A$ in the same way we denote elements in $kH$, e.g.\ we will write $\bar{d} = 0$.  As $x \in [\A, \A]\A$, any base element in the support of $x$ features a factor $\overline{c}$ and has weight at least $3$. Moreover, as a consequence of Lemma \ref{lem:Centralizerx}, the element $\bar{c} +x$ is central, so in particular 
\[(\bar{c}+x)^h = \bar{c}+x.\]
From the equalities $[c,h] = c^{p^t}$ and $[g,h] = c$ and Lemma~\ref{lemma:BasicJenningBasisFormulas} we derive
\begin{equation}\label{eq:conj-x}
\bar{c}^h = \bar{c} + \overline{c^{p^t}} + \bar{c}\overline{c^{p^t}}
\textup{ and }
\bar{g}^h = \bar{g} + \bar{c} + \bar{g}\bar{c},
\end{equation}
from which it follows that
\begin{equation}\label{eq:x^h-x}
x^h - x = \bar{c}-\bar{c}^h =   -\overline{c^{p^t}} - \bar{c}\overline{c^{p^t}}.
\end{equation}
We will show that \eqref{eq:x^h-x} cannot hold. We start by showing that, to this aim, we can restrict our computations to a specific subset of the support elements of $x$. We can ignore central base elements, because they do not contribute to $x^h-x$. 
Moreover, if a support element $y$ contains a factor from $\{ \bar{h}, \overline{h^p},\ldots,\overline{h^{p^t}}\}$, then, as a consequence of \eqref{eq:conj-x} and the choice of the Jennings basis, such factor will also be present in any support element of $y^h$. %This is due to the above conjugation formulas and the fact that we chose our Jennings tuple in $H$ in a way such that $h^{p^\ell}$ precedes $g^{p^\ell}$ for any $\ell \leq t$.
%Moreover by Lemma~\ref{lem:Centralizerx} no non-central base element in the support of $x$ has a factor from $\{ \bar{h}, \overline{h^p},\ldots,\overline{h^{p^t}}\}$, as this would contradict $0 = [x, \bar{g}]$. \ChLeo{Would it? Even if yes, this seems non-trivial. Maybe go back to argument from lockdown times if necessary} 
Finally, assume that $y$ is in the support of $x$ and contains a central factor from $\overline{H^{p^{t+1}}}$ which does not lie in $\overline{\gamma_2(H)}$. Then the projection of $y^h-y$ onto the linear subspace spanned by $\overline{c^{p^t}}$ and $\bar{c}\overline{c^{p^t}}$ is $0$ as this central factor cannot be eliminated by conjugation. So considering $y$ cannot influence the satisfiability of \eqref{eq:x^h-x}.

To show that $x$ cannot satisfy \eqref{eq:x^h-x}, we assume without loss of generality that each element $y$ in the support of $x$ is of the form $\bar{g}^i\bar{c}^j$, where $i$ and $j$ are non-negative integers with $j > 0$. The positivity of $j$ follows from $x$ being in $[\A,\A]\A$ and we have used the intercommutativity of $\bar{g}$ and $\bar{c}$ to simplify the notation.
Each such support element $y$ can be written as 
\[y = \overline{g}^{i_0}\overline{g^p}^{i_1}...\overline{g^{p^t}}^{i_t}\overline{c}^{j_0}\overline{c^p}^{j_1}...\overline{c^{p^t}}^{j_t} \]
where $0\leq i_0,\ldots,i_t,j_0,\ldots, j_t < p$ and
\[i_0+ pi_1 + \ldots +p^t i_t = i \text{ and } j_0 + p j_1 +\ldots +p^t j_t = j\]
are the $p$-adic expansions of $i$ and $j$.
Following this notation, we compute
\begin{equation}\label{eq:y^h}
y^h = (\overline{g} + \overline{c} +\overline{g} \ \overline{c})^{i_0}...(\overline{g^{p^t}} + \overline{c^{p^t}} + \overline{g^{p^t}} \ \overline{c^{p^t}})^{i_t}(\overline{c} + \overline{c^{p^t}} + \overline{c}\overline{c^{p^t}})^{j_0} \overline{c^p}^{j_1}...\overline{c^{p^t}}^{j_t}.  
\end{equation}
Now if a certain $y = \bar{g}^i\bar{c}^j$ is in the support of $x$ and the support of $y^h - y$ contains an element $\tilde{y}$ different from $\overline{c^{p^t}}$ and  $\bar{c}\overline{c^{p^t}}$, then there must be another support element $y'$ in $x$ such that $(y')^h-y'$ also contains $\tilde{y}$ in its support. Relying on the last observation, we will show that $y$ is not in the support of $x$ for $1\leq i,j \leq p^t-1$. Note that it is sufficient to show that no $y$ satisfying these conditions on $i$ and $j$ is an element in the support of $x$ as only in these cases $y^h-y$ can contain $\overline{c^{p^t}}$ as an element in the support and $\overline{c^{p^t}}$ is an element necessary for \eqref{eq:x^h-x} to hold. Our restrictions, in particular, imply $i_t = j_t = 0$ in \eqref{eq:y^h}.

We will prove the impossibility of any such $y$ coming up in the support of $x$ by induction on triples of the form $(\ell, j, r)$, where $i = p^\ell(sp-r)$ for some integer $s$, ordered lexicographically. Observe that the entries of such triples range over 
$$0\leq\ell\leq t-1, \quad 1\leq j \leq p^t-1, \quad 1\leq r \leq p-1.$$ %We view the triples ordered lexicographically. 
When considering $y^h-y$, note that the elements in the support are those which can be gotten from multiplying explicitly in \eqref{eq:y^h} with the exception of the element $\bar{g}^{i_0}\ldots\overline{g^{p^t}}^{i_t}\bar{c}^{j_0}\ldots\overline{c^{p^t}}^{j_t}=y$. In particular, working with the factor $(\overline{c} + \overline{c^{p^t}} + \overline{c}\overline{c^{p^t}})^{j_0} \overline{c^p}^{j_1}...\overline{c^{p^t}}^{j_t}$ in \eqref{eq:y^h}, we derive: 

\begin{center}
{\bf Trick:} If $y^h-y$ contains an element $\bar{g}^{i'}\bar{c}^{j'}$ in its support such that $j' < j+p^t-1$, then at least one of the factors in $\bar{c}^{j'}$ is coming from the product 
$$(\overline{g} + \overline{c} +\overline{g} \ \overline{c})^{i_0}...(\overline{g^{p^t}} + \overline{c^{p^t}} + \overline{g^{p^t}} \ \overline{c^{p^t}})^{i_t}.$$
\end{center}

\noindent
Note that, if $j'$ is of the form $j' \leq j+p^\ell$ for some $\ell$ belonging to the induction range, then the Trick is always applicable. 

For the induction base, assume $(\ell,j,r)=(0,1,1)$. Then, for some $s$, we have $i=sp -1$ and the element $\overline{g}^{sp -1}\overline{c}^2$ is in the support of $y^h-y$ and $y$ is the only element with this property. Indeed, let $y' = \bar{g}^{i'}\bar{c}^{j'}$ be such that $\overline{g}^{sp -1}\overline{c}^2$ is in the support of $y'^h-y'$ and compute 
\begin{equation}\label{eq:y'}
 (y')^h = (\overline{g} + \overline{c} +\overline{g} \ \overline{c})^{i_0'}...(\overline{g^{p^t}} + \overline{c^{p^t}} + \overline{g^{p^t}} \ \overline{c^{p^t}})^{i_t'}(\overline{c} + \overline{c^{p^t}} + \overline{c}\overline{c^{p^t}})^{j_0'} \overline{c^p}^{j_1}...\overline{c^{p^t}}^{j_t'}.  
 \end{equation}
As the support of $(y')^h-y'$ contains an element featuring the factor $\bar{c}$ exactly twice, we deduce that $1 \leq j' \leq 2$. Moreover, the factor $(\overline{g^p} + \overline{c^p} +\overline{g^p} \ \overline{c^p})^{i_1'}...(\overline{g^{p^t}} + \overline{c^{p^t}} + \overline{g^{p^t}} \ \overline{c^{p^t}})^{i_t'}$ does not contribute a factor $\bar{c}$, as the smallest available factor is $\bar{c}^p$. % (note that the indices start with $1$, not $0$).  
The factor  $\overline{g}^{sp -1}$ can only be obtained from the subproduct $(\overline{g} + \overline{c} +\overline{g} \ \overline{c})^{i_0'}...(\overline{g^{p^t}} + \overline{c^{p^t}} + \overline{g^{p^t}} \ \overline{c^{p^t}})^{i_t'}$ and, being $(\overline{g} + \overline{c} +\overline{g} \ \overline{c})^{i_0'}$ the only candidate able to contribute a factor $\bar{c}$, the index $i'$ must equal $ps-1$. Combining the Trick with the fact that one factor $\bar{c}$ is contributed by $(\overline{c} + \overline{c^{p^t}} + \overline{c}\overline{c^{p^t}})^{j_0'} \overline{c^p}^{j_1}...\overline{c^{p^t}}^{j_t'}$, we conclude that $j'=1$.

We proceed with the induction step. Write $i = p^\ell(sp-r)$ and assume that the triples satisfying $(\ell',j',r')<(\ell,j,r)$ have been excluded. 
We claim that the element $\overline{g}^{p^\ell(sp -r)}\overline{c}^{j+p^\ell}$ is in the support of $y^h-y$ and assume that this element also appears in the support of $y'^h-y'$ for some $y'$, presented as in \eqref{eq:y'} and associated to the triple $(\ell',j',r')$. It follows from the induction hypothesis that $i_0'=\ldots=i_{\ell-1}'=0$. % as otherwise we would have $\ell' < \ell$. 
We note that, for the factor $\bar{g}^{p^\ell(ps-r)}$ to be realizable, 
it is necessary that $i_\ell' \neq 0$ and $\ell' = \ell$. The induction hypothesis yields $j' \geq j$ and the factor 
\[(\overline{g^{p^{\ell+1}}} + \overline{c^{p^{\ell+1}}} +\overline{g^{p^{\ell+1}}} \ \overline{c^{p^{\ell+1}}})^{i_{\ell+1}'}...(\overline{g^{p^t}} + \overline{c^{p^t}} + \overline{g^{p^t}} \ \overline{c^{p^t}})^{i_t'}\]
cannot contribute to the factor $\bar{c}^{j'}$ as its minimum contribution would be $\bar{c}^{p^{\ell+1}}$. On the other hand, since $j' \leq j +p^\ell$, the Trick ensures that the  product $(\overline{g} + \overline{c} +\overline{g} \ \overline{c})^{i_0'}...(\overline{g^{p^t}} + \overline{c^{p^t}} + \overline{g^{p^t}} \ \overline{c^{p^t}})^{i_t'}$ contributes the factor $\bar{c}^{j'}$. The minimum contribution is $\bar{c}^{p^\ell}$, as $i_0'=\ldots=i_{\ell-1}' = 0$. The maximal possible contribution 
being
$\bar{c}^{j+p^\ell}$, we conclude that $j = j'$.  It follows from the induction hypothesis that $r' \geq r$. In particular, to obtain the factor $\overline{g}^{p^\ell(sp -r)}$, the product $(\overline{g} + \overline{c} +\overline{g} \ \overline{c})^{i_0'}...(\overline{g^{p^t}} + \overline{c^{p^t}} + \overline{g^{p^t}} \ \overline{c^{p^t}})^{i_t'}$ has to contribute the maximal amount of possible factors $\bar{g}$ at the same time, i.e.\ contributing only $p^\ell$ times the factor $\bar{c}$. We deduce that $r' = r$.

In order to conclude the proof, we shall make sure that the elements of the form $\overline{g}^{p^\ell(sp -r)}\overline{c}^{j+p^\ell}$ used in the arguments above are not congruent to $0$ modulo $\mathcal{L}$. None of such elements is central as an element of $kG$, as $p^\ell(sp-r) < p^t$, and so they do not lie in $\mu_p^0(Z)$, if they do not lie in $\mathcal{K}$ itself. We will consider their weights to determine if they lie in $\mathcal{K}$, taking into account that an element $\overline{c^{p^\ell}}$ might have smaller weight than $\overline{c^{p^\ell}d^u}$ for some integer $u$, even these two elements are congruent modulo $\mathcal{K}$. 

Among the elements in consideration, the ones of shape $\overline{g}^{p^\ell(sp -1)}\overline{c}^{j+p^\ell}$ have biggest weight, so we will ignore the case $r > 1$.
As $p^\ell(sp-1) < p^{t+1}$, Lemma~\ref{lemma:JenningsBasisChoice} implies that 
$$\wt\left(\bar{g}^{p^\ell(sp-1)}\right) = p^\ell(sp-1).$$ 
Moreover, the weight of $\bar{c}$ being $2$, the maximal possible weight of an element $\overline{g}^{p^\ell(sp -1)}\overline{c}^{j+p^\ell}$ is realized when $\ell = t-1$ and $j = p^t-1$. We define 
\[z = \overline{g}^{p^t-p^{t-1}}\overline{c}^{p^t+p^{t-1}-1},\]
which comes up when, in the induction process, we exclude $\overline{g}^{p^t-p^{t-1}}\overline{c}^{p^t-1}$. 
We claim that $\wt(z) < 2w$. Erasing the factor $\overline{c^{p^t}}$ of weight not larger than $w$, Lemma~\ref{lemma:beta} yields the translation of the claim into $\wt\left(\overline{g}^{p^t-p^{t-1}}\overline{c}^{p^{t-1}-1}\right) < w$. As noted above, the weight of $\overline{g}^{p^t-p^{t-1}}$ is $p^t-p^{t-1}$. Let now $u$ be an integer such that, for any integer $u_0$, one has $w' = \wt\left(c^{p^{t-1}}d^u \right) \geq \wt \left(c^{p^{t-1}}d^{u_0} \right)$. It follows that $w \geq p w'$ and, the weight of $\bar{c}$ being $2$,  that $w' \geq 2p^{t-1}$. 
From the last observations we derive 
\begin{align*}
\wt\left(\overline{g}^{p^t-p^{t-1}}\overline{c}^{p^{t-1}-1}\right)&\leq \wt\left(\overline{g}^{p^t-p^{t-1}}\overline{c}^{p^{t-1}}\right)\leq \wt\left(\overline{g}^{p^t-p^{t-1}}\overline{c^{p^{t-1}}d^u}\right)=p^{t}-p^{t-1}+w' \\
& \leq p^{t}-p^{t-1}+w-(p-1)w' \leq p^{t}-p^{t-1}+w-(p-1)2p^{t-1} \\
& \leq w-(p^t -p^{t-1})<w
\end{align*}
and so the claim is proven.
\end{proof}

\begin{remark}
Theorem~\ref{th:2gen-class3-orderG3} can be used to solve the (MIP) in some cases that cannot be handled by the known group theoretical invariants. For instance, if $G = \SG(3^7,19)$ and $H$ is a group not isomorphic to $G$ but sharing all group theoretical invariants with $G$, then $H \cong \SG(3^7, 43)$. One finds that $|\gamma_3(G)| = 3$ while $|\gamma_3(H)| = 3^2$, so Theorem~\ref{th:2gen-class3-orderG3} yields $kG \not\cong kH$. Other examples of groups for which group theoretical invariants are not sufficient to solve the (MIP), but are when upgraded with Theorem~\ref{th:2gen-class3-orderG3}, are
\begin{align*}
\SG(3^7,23)&, \SG(3^7,45), \SG(3^7,27), \SG(3^7,47), \\ \SG(5^7,19)&, \SG(5^7, 21), \SG(5^7, 77), \SG(5^7, 87), \SG(5^7, 30086), \\ \SG(7^7, 23)&, \SG(7^7, 25), \SG(7^7, 111), \SG(7^7, 125), \SG(7^7, 104602).
\end{align*}
It seems like still some work will be required to understand this evidence in a generic sense. Some of the groups listed above for which the third terms in their lower central series' have the same size can only be distinguished using the so-called Roggenkamp parameter $\sum_{g^G} \log_p(|C_G(g)/\Phi(C_G(g))| )$, where the sum runs over conjugacy classes of $G$. This invariant is far from easy to compute by hand in a general situation.
\end{remark}

\section{The class of obelisks}\label{sec:Obelisks}
In computer-aided investigations of the (MIP) such as \cite{Wur93, Eic08, MM20} one is interested in determining the smallest number $m$ such that, given two non-isomorphic groups $G$ and $H$, one has 
$I(kG)/I(kG)^m \not\cong I(kH)/I(kH)^m$. 
It follows directly from the theory of Jennings that 
$$G/D_\ell(G) \cong H/D_\ell(H) \Longrightarrow I(kG)/I(kG)^\ell \cong I(kH)/I(kH)^\ell.$$ 
The following questions are based on computational evidence.

\begin{qs}\label{qs:JenningsBound}
Let $H$ be a $p$-group not isomorphic to $G$ satisfying $|H|=|G|$ and let $\ell$ be the maximal integer such that $G/D_\ell(G) \cong H/D_\ell(H)$. Are the following statements true? 
\begin{enumerate}[label=$(\arabic*)$]
\item It holds that $I(kG)/I(kG)^{2\ell+1} \not\cong I(kH)/I(kH)^{2\ell+1}$ \cite{BKRW99}.
\item If $p$ is odd and $m= \max\{\ell, 2\ell-\frac{p-1}{2}\}$, then $I(kG)/I(kG)^{m+1} \not\cong I(kH)/I(kH)^{m+1}$ \cite{MM20}.
\end{enumerate}   
\end{qs}
Question \ref{qs:JenningsBound}(1) has been answered negatively in \cite{MM20} for groups of order $2^8$. Question \ref{qs:JenningsBound}(2) is open. 
There are  very few bins of groups
of order $5^6$, which share all known group theoretical invariants and for which the upper bound given in Question \ref{qs:JenningsBound}(2) is sharp \cite[Section 3.3]{MM20}. A family that is highlighted in this process is that of the so-called  \emph{$p$-obelisks} (see Definition \ref{def:obelisk} below). These groups were first investigated in \cite{Bla61} and later in \cite{Sta17}. As a first step towards understanding the modular group algebras of $p$-obelisks and solving the (MIP) for this class of groups, we show in this section that the property of being a $p$-obelisk is an invariant of the group algebra. 
Until the end of the present section, \underline{assume} that $p>3$.% be a prime number.

\begin{definition}\label{def:obelisk} 
A \emph{$p$-obelisk} is a finite non-abelian $p$-group $\cor{O}$ with $|\cor{O}:\gamma_2(\cor{O})| = p^2$ and $\cor{O}^p = \gamma_3(\cor{O})$.
\end{definition}

\noindent
The next theorem is the main result of this section, which will be proven after Lemma~\ref{lemma_DimSubOb}. 

\begin{theorem}\label{theo_IsOb}
Let $\mathcal{O}$ be a $p$-obelisk such that $kG \cong k\mathcal{O}$. Then $G$ is a $p$-obelisk.
\end{theorem}

\noindent
The following proposition collects known properties of $p$-obelisks; see for example \cite[Lemmas 322-3]{Sta17}
%\cite[Lemmas 10.4-5]{Sta17} 
or \cite[Theorem 4.3 and its proof]{Bla61} (this is also to be found in \cite[III, Satz 17.9]{Hup67}).

\begin{proposition}\label{prop_ObProps}
Let $G$ be a $p$-obelisk of class $c$ and $i$ a positive integer. Then the following hold: 
\begin{enumerate}[label=$(\arabic*)$]
\item the quotient $\gamma_i(G)/\gamma_{i+1}(G)$ is elementary abelian of rank \[
\dim_{k}(\gamma_i(G)/\gamma_{i+1}(G))=\begin{cases}
2 & \textup{if } i\leq c-1 \textup{ and $i$ odd},\\
1 & \textup{if } i\leq c \textup{ and $i$ even}.
\end{cases}
\] 
\item one has $\gamma_i(G)^p = \gamma_{i+2}(G)$.
\end{enumerate}
\end{proposition}

\begin{lemma}\label{lemma_DimSubOb}
Let $\cor{O}$ be a $p$-obelisk. Let $n$ be a positive integer and write $n = ap^\ell + b$ where $\ell \geq 0$, $1 \leq a \leq p-1$ and $0 \leq b < p^{\ell}$. Then 
$D_n(\cor{O}) = \gamma_{m(n)}(\cor{O})$ where
$$ m(n) = \left\{\begin{array}{lll} 2\ell + 1, \ \ \text{if} \ a = 1 \ \text{and} \ b = 0, \\ 2\ell + 3, \ \ \text{if} \ a > 2 \ \text{or}  \ a = 2 \ \text{and} \ b \geq 1, \\ 2\ell +2, \ \ \text{otherwise} \end{array}\right. $$
\end{lemma}
\begin{proof}
We work by induction on $n$ and using the iterative formula for dimension subgroups. We clearly have $D_1(\cor{O}) = \gamma_1(\cor{O})$. Moreover, from $p \geq 5$ we compute $D_2(\cor{O}) = \gamma_2(\cor{O})\cor{O}^p$ and $D_3(\cor{O}) = \gamma_3(\cor{O}) \cor{O}^p$. The group $\cor{O}$ being a $p$-obelisk, we deduce $D_2(\cor{O}) = \gamma_2(\cor{O})$ and $D_3(\cor{O}) = \gamma_3(\cor{O})$. We have proven the claim when $n$ equals $1$, $2$ or $3$.

Assume now that $n \geq 4$ and that, for each $n'<n$, the equality $D_{n'}(\cor{O})=\gamma_{m(n')}(\cor{O})$ holds. We write 
\[
n = ap^\ell + b \textup{ with } \ell \geq 0, \ 1 \leq a \leq p-1, \ 0 \leq b < p^{\ell}.
\]
Let first $a = 1$ and $b = 0$; then $n\geq 4$ implies $\ell\neq 0$ and, in particular, we have $2\ell +1 \leq p^\ell=n$. Note that, by Proposition~\ref{prop_ObProps}(2), we have $\gamma_{2\ell-1}(\cor{O})^p = \gamma_{2\ell+1}(\cor{O})$. Thanks to the inductive formula for the dimension subgroups, cf. Section~\ref{sec:Ideals}, the induction hypothesis yields
$$
D_n(\cor{O}) = D_{ap^{\ell-1}}(\cor{O})^p \gamma_n(\cor{O}) = \gamma_{2(\ell-1)+1}(\cor{O})^p \gamma_n(\cor{O}) = \gamma_{2\ell + 1}(\cor{O}).
$$
Now let $a > 2$ or $a = 2$ and $b \geq 1$. The inductive formula for the dimension subgroups combined with Proposition~\ref{prop_ObProps} and the fact that $n\geq 4$ yields
$$D_n(\cor{O}) = D_{ap^{\ell-1} + \ceil{\frac{b}{p}} }(\cor{O})^p \gamma_n(\cor{O}) = \gamma_{2(\ell-1)+3}(\cor{O})^p \gamma_n(\cor{O}) = \gamma_{2\ell + 3}(\cor{O}).$$
We next assume that $a = 1$ and $b \geq 1$ or $a = 2$ and $b = 0$. The combination of Proposition~\ref{prop_ObProps}, the induction hypothesis, and the fact that $p\geq 5$ yields
$$D_n(\cor{O}) = D_{ap^{\ell-1} + \ceil{\frac{b}{p}}}(\cor{O})^p \gamma_n(\cor{O}) = \gamma_{2(\ell-1)+2}(\cor{O})^p \gamma_n(\cor{O}) = \gamma_{2\ell + 2}(\cor{O}).$$
\end{proof}

\noindent
We give here the proof of Theorem \ref{theo_IsOb}. To this end and until the end of the proof, let $\mathcal{O}$ be a $p$-obelisk such that $kG \cong k\mathcal{O}$. We show that $G$ is a $p$-obelisk.

The (MIP) is solved for orders dividing $p^5$ \cite{SS96p5} so
we assume without loss of generality that $|G| = |\mathcal{O}| \geq p^6$. Then, by Proposition \ref{prop_ObProps}(1), the nilpotency class of $\mathcal{O}$ is at least 4.
From Section~\ref{sec:invariants} we know that $G$ and $\mathcal{O}$ have the same order and also that $G/\gamma_2(G) \cong \mathcal{O}/\gamma_2(\mathcal{O})$; the last quotients are elementary abelian of order $p^2$ by Proposition \ref{prop_ObProps}(1). %In particular $[G:\gamma_2(G)] = p^2$. 
%Moreover, by \cite[III, 1.11 Hilfsatz]{Hup67} the group 
In particular $\gamma_2(G)/\gamma_3(G)$ and $\gamma_2(\cor{O})/\gamma_3(\cor{O})$ are cyclic 
%and by \cite[III, 2.13 Satz]{Hup67} it has exponent $p$, i.e. $\gamma_2(G)/\gamma_3(G) \cong C_p$ 
of order $p$
and hence $|\gamma_2(G):\gamma_3(G)| = |\gamma_2(\mathcal{O}):\gamma_3(\mathcal{O})|$. 
Now, by Section~\ref{sec:invariants}(4) and Lemma~\ref{lemma_DimSubOb} we have 
\[D_1(G)/D_2(G) \cong D_1(\mathcal{O})/D_2(\mathcal{O}) = \mathcal{O}/\gamma_2(\mathcal{O}).\]
From $D_2(G) = \gamma_2(G)G^p$ we then obtain $G^p \leq \gamma_2(G)$ and $D_2(G)=\gamma_2(G)$. In the same way we get 
$$D_2(G)/D_3(G) \cong D_2(\mathcal{O})/D_3(\mathcal{O}) = \gamma_2(\mathcal{O})/\gamma_3(\mathcal{O}). $$
Both $D_3(G) = \gamma_3(G)G^p$ and $\gamma_3(G)$ being subgroups of index $p$ of $D_2(G) = \gamma_2(G)$,  we get $G^p \subseteq \gamma_3(G)$ and, in particular, that $D_3(G)=\gamma_3(G)$. 
%From the Sandling quotient in Proposition~\ref{prop_GroupInv} and Proposition~\ref{prop_ObProps} we obtain
%$$G/(G_2)^pG_3 \cong \mathcal{O}/(\mathcal{O}_2)^p \mathcal{O}_3 \cong \mathcal{O}/\mathcal{O}_4\mathcal{O}_3 = \mathcal{O}/\mathcal{O}_3, $$
%so $|G/(G_2)^pG_3| = p^3$ and as $G_2/G_3$ is cyclic we get $G_2^p \subseteq G_3$ and $G_2/G_3 \cong C_p$.
%Together with $D_3(G) = G_3G^p$ this implies $G^p \leq G_3$. \ChLeo{But why? I don't have more notes, yours say "bilinear magic"} Moreover $[G_2:G_3] = [\mathcal{O}_2:\mathcal{O}_3]$.
We claim that, for each integer $n\geq 3$, the following holds:
\begin{itemize}
\item[(i)] $\gamma_3(G) = G^p\gamma_n(G)$
\item[(ii)] $\gamma_i(G) = D_j(G)$ for all $1 \leq i \leq n$ and any $j$ satisfying $\gamma_i(\mathcal{O}) = D_j(\mathcal{O})$.
\end{itemize}
To prove the last claim, we will work by induction on $n$. For \emph{this proof only}, we will use the bar notation with a different meaning from that introduced in Section \ref{sec:alg not}.

Assume $n = 3$. We have already proven (i), so we show (ii). By Lemma~\ref{lemma_DimSubOb}, one has $\gamma_3(\mathcal{O}) = D_j(\mathcal{O})$ if and only if $3 \leq j \leq p$. When $i\in\graffe{1,\ldots,p}$, we get from Section~\ref{sec:invariants}(4) that 
\[D_i(\mathcal{O})/D_{i+1}(\mathcal{O}) \cong D_i(G)/D_{i+1}(G)\]
and so $\gamma_3(G) = D_j(G)$ if and only if $3 \leq j \leq p$. This proves the base case. We now assume that the claim holds for $n\geq 3$ and prove it for $n+1$. 

Assume that $n$ is even. 
Set $\overline{G} = G/\gamma_{n+1}(G)$. We claim that $\overline{G}$ is a $p$-obelisk. Indeed, applying in this order the induction hypothesis on (i), then on (ii), the formula for dimension subgroups, and then again (ii), one gets 
$$\gamma_3(\overline{G}) = \overline{G}^p\gamma_n(\overline{G}) = \overline{G}^p D_{2p^{\frac{n-2}{2}}}(\overline{G}) = \overline{G}^p (D_{2p^{\frac{n-4}{2}}}(\overline{G}))^p \gamma_{2p^{\frac{n-2}{2}}}(\overline{G}) = \overline{G}^p \gamma_{n-2}(\overline{G})^p = \overline{G}^p. $$
We have proven that $\overline{G}$ is a $p$-obelisk, equivalently $\gamma_3(G) = G^p\gamma_{n+1}(G)$ and (i) holds.
%By the isomorphism of subsequent quotients of dimension subgroups we have $D_j(G) = D_{j'}(G)$ for each $j$ and $j'$ such that $\mathcal{O}_{n+1} = D_j(\mathcal{O}) = D_{j'}(\mathcal{O})$. By Lemma~\ref{lemma_DimSubOb} such a $j$  \ChMima{[Why is there just one $j$ and why can the case $n+1=2l+3$ not happen? This proof is a little difficult to follow with all the indices so it could be helpful to give some more detail maybe.]}\ChLeo{There are many such $j$'s of course, but we pick a fixed one. It should probably be rewritten. I will make an attempt on the whole paragraph below.} is given by $j = p^{\frac{n}{2}}$. By Lemma~\ref{lemma_DimSubIndcutive} and induction we obtain
%$D_{p^{\frac{n}{2}}}(G) = (D_{p^{\frac{n-2}{2}}}(G))^p G_{p^{\frac{n}{2}}} = (G_{n-1})^p G_{p^{\frac{n}{2}}}$. From that equality by induction we obtain $G_n/(G_{n-1})^pG_{p^{\frac{n}{2}}} \cong \mathcal{O}_n/\mathcal{O}_{n+1} \cong C_p$. As $\overline{G}$ is an obelisk, Proposition~\ref{prop_ObProps} yields $\overline{G}_{n-1}^p \leq \overline{G}_{n+1} = 1$. So $G_{n-1}^p G_{p^{\frac{n}{2}}} \leq G_{n+1} \leq G_n$. As $[G_n:G_{n-1}^p G_{p^{\frac{n}{2}}}] = p$ by the above we get that either $G_n = G_{n+1}$ or $G_{n+1} = G_{n-1}^p G_{p^{\frac{n}{2}}}$. If $G_n = G_{n+1}$ then $G_n = 1$ and as by induction $G_3 = G^pG_n$ we would obtain that $G$ is an obelisk. On the other hand $G_{n-1}^p G_{p^{\frac{n}{2}}} = G_{n+1}$ together with the reasoning above and $[G_n:G_{n+1}] = p = [\mathcal{O}_n:\mathcal{O}_{n+1}]$ imply the induction claims (ii) and (iii) for $n+1$.
We continue with (ii). By Lemma~\ref{lemma_DimSubOb} ,we have $D_{p^{\frac{n}{2}}}(\mathcal{O}) = \gamma_{n+1}(\mathcal{O})$ and so the induction hypothesis on (ii) yields
$$ D_{p^{\frac{n}{2}}}(G) = D_{p^{\frac{n-2}{2}}}(G)^p \gamma_{p^{\frac{n}{2}}}(G) = \gamma_{n-1}(G)^p \gamma_{p^{\frac{n}{2}}}(G). $$
Note that, by Lemma~\ref{lemma_DimSubOb}, we have that $\gamma_n(\mathcal{O}) = D_{2p^{\frac{n}{2}-1}}(\mathcal{O})$ and, for each $2p^{\frac{n}{2}-1} +1 \leq j \leq p^\frac{n}{2}$, that $\gamma_{n+1}(\mathcal{O}) = D_j(\mathcal{O})$. It follows from the combination of Section~\ref{sec:invariants}(3) with the induction hypothesis on (ii) that $\gamma_n(G) = D_{2p^{\frac{n}{2}-1}}(G)$ and, for each $2p^{\frac{n}{2}-1} +1 \leq j,j' \leq p^\frac{n}{2}$, that $D_j(G) = D_{j'}(G)$. 
Hence we get
\begin{align}\label{eq_evenCase}
\gamma_n(G)/\gamma_{n-1}(G)^p \gamma_{p^\frac{n}{2}}(G) = D_{2p^{\frac{n}{2}-1}}(G)/D_{p^{\frac{n}{2}}}(G) \cong D_{2p^{\frac{n}{2}-1}}(\mathcal{O})/D_{p^{\frac{n}{2}}}(\mathcal{O}) =  \gamma_n(\mathcal{O})/\gamma_{n+1}(\mathcal{O}) 
\end{align}
with $|\gamma_n(\mathcal{O}):\gamma_{n+1}(\mathcal{O})|\leq p$.
As $\overline{G}$ is an obelisk, Proposition~\ref{prop_ObProps} yields $\gamma_{n-1}(\overline{G})^p \subseteq \gamma_{n+1}(\overline{G}) = 1$ from which it follows that $\gamma_{n-1}(G)^p \gamma_{p^{\frac{n}{2}}}(G) \subseteq \gamma_{n+1}(G) \subseteq \gamma_n(G)$. As a consequence of \eqref{eq_evenCase} we have that $|\gamma_n(G):\gamma_{n-1}(G)^p \gamma_{p^{\frac{n}{2}}}(G)| \leq p$ and so either $\gamma_n(G) = \gamma_{n+1}(G)$ or $\gamma_{n+1}(G) = \gamma_{n-1}(G)^p \gamma_{p^{\frac{n}{2}}}(G)$. If $\gamma_n(G) = \gamma_{n+1}(G)$, then $\gamma_n(G) = 1$ and, as by induction $\gamma_3(G) = G^p\gamma_n(G)$, we obtain that $G$ is an obelisk. Assume, on the other hand that $\gamma_{n-1}(G)^p \gamma_{p^{\frac{n}{2}}}(G) = \gamma_{n+1}(G)$. Then $\gamma_{n+1}(G) = D_{p^\frac{n}{2}}(G)$ together with the reasoning above implies (ii) for $n+1$. 

Assume now that $n$ is odd. 
Set $\overline{G} = G/\gamma_{n+1}(G)$, which we can show to be an obelisk via computing
\begin{align}\label{eq_Gbarob}
\gamma_3(\overline{G}) = \overline{G}^p\gamma_n(\overline{G}) = \overline{G}^p D_{p^{\frac{n-1}{2}}}(\overline{G}) = \overline{G}^p (D_{p^{\frac{n-3}{2}}}(\overline{G}))^p \gamma_{p^{\frac{n-1}{2}}}(\overline{G}) = \overline{G}^p \gamma_{n-2}(\overline{G})^p = \overline{G}^p. 
\end{align}
In particular the equality $\gamma_3(G) = G^p \gamma_{n+1}(G)$ holds.
Now, by Lemma~\ref{lemma_DimSubOb} we have $D_{2p^{\frac{n-1}{2}}}(\mathcal{O}) = \gamma_{n+1}(\mathcal{O})$. Then, by the iterative formula for dimension subgroups and induction hypothesis on (ii), we have
$$D_{2p^{\frac{n-1}{2}}}(G) = D_{2p^{\frac{n-3}{2}}}(G)^p \gamma_{2p^{\frac{n-1}{2}}}(G) = \gamma_{n-1}(G)^p \gamma_{2p^{\frac{n-1}{2}}}(G). $$ 
It follows from Lemma~\ref{lemma_DimSubOb} that $\gamma_n(\mathcal{O}) = D_{p^{\frac{n-1}{2}}}(\mathcal{O})$ and, for each $p^{\frac{n-1}{2}} +1 \leq j \leq 2p^\frac{n-1}{2}$, that $\gamma_{n+1}(\mathcal{O}) = D_j(\mathcal{O})$. Section~\ref{sec:invariants}(4) and the induction hypothesis on (ii) then yield that $\gamma_n(G) = D_{p^{\frac{n-1}{2}}}(G)$ and, for each $p^{\frac{n-1}{2}} +1 \leq j, j' \leq 2p^\frac{n-1}{2}$, that  $D_j(G) = D_{j'}(G)$. 
The following hold:
\begin{align}\label{eq_oddCase}
\gamma_n(G)/\gamma_{n-1}(G)^p \gamma_{2p^\frac{n-1}{2}}(G) = D_{p^{\frac{n-1}{2}}}(G)/D_{2p^{\frac{n-1}{2}}}(G) \cong D_{p^{\frac{n-1}{2}}}(\mathcal{O})/D_{2p^{\frac{n-1}{2}}}(\mathcal{O}) =  \gamma_n(\mathcal{O})/\gamma_{n+1}(\mathcal{O}) 
\end{align}
where $\gamma_n(\mathcal{O})/\gamma_{n+1}(\mathcal{O})$ is elementary abelian of rank at most $2$.
%We have again $D_j(G) = D_{j'}(G)$ for any $j$ and $j'$ which satisfy $\mathcal{O}_{n+1} = D_j(\mathcal{O}) = D_{j'}(\mathcal{O})$. Such a $j$ is given by $2p^{\frac{n-1}{2}}$. By Lemma~\ref{lemma_DimSubIndcutive} and induction we get
%$$D_{2p^{\frac{n-1}{2}}}(G) = D_{2p^{\frac{n-3}{2}}}(G)^p G_{2p^{\frac{n-1}{2}}} = G_{n-1}^p G_{2p^{\frac{n-1}{2}}}. $$ 
%Hence $\mathcal{O}_n/\mathcal{O}_{n+1} \cong G_n/G_{n-1}^p G_{2p^{\frac{n-1}{2}}}$. 
Since $\overline{G}$ is a $p$-obelisk, we have $\gamma_{n-1}(\overline{G})^p = \gamma_{n+1}(\overline{G}) = 1$ and so $\gamma_{n-1}(G)^p \gamma_{2p^{\frac{n-1}{2}}}(G) \subseteq \gamma_{n+1}(G) \subseteq \gamma_n(G)$.
If $\gamma_n(G) = \gamma_{n+1}(G)$, then $\gamma_3(G) = G^p$ and $G$ is a $p$-obelisk. If $p = |\gamma_n(G):\gamma_{n+1}(G)| = |\gamma_n(G):\gamma_{n-1}(G)^p \gamma_{2p^{\frac{n-1}{2}}}(G)|$, then $|\gamma_n(\mathcal{O}):\gamma_{n+1}(\mathcal{O})| = p$ which, together with Proposition~\ref{prop_ObProps}, implies $\gamma_{n+1}(\mathcal{O}) = 1$. As $G$ and $\mathcal{O}$ have the same order, we get that $\gamma_{n+1}(G) = 1$ and the equality $\gamma_3(G) = G^p \gamma_{n+1}(G) = G^p$ yields that $G$ is a $p$-obelisk.
Assume now that $|\gamma_n(G): \gamma_{n-1}(G)^p \gamma_{2p^{\frac{n-1}{2}}}(G)| = p^2$. If also $|\gamma_n(G): \gamma_{n+1}(G)| = p^2$ then $D_{2p^\frac{n-1}{2}}(G) = \gamma_{n-1}(G)^p \gamma_{2p^{\frac{n-1}{2}}}(G) = \gamma_{n+1}(G)$ and (ii) follows from a similar argument as before. So the only possibility left is $|\gamma_n(G):\gamma_{n+1}(G)| = p$. Set $\overline{G} = G/\gamma_{n-1}(G)^p \gamma_{2p^{\frac{n-1}{2}}}(G)$. As $\gamma_{n+1}(\overline{G})$ has order $p$, we get $\gamma_{n+2}(\overline{G}) = 1$. Moreover, for each $p \geq 5$ and $n \geq 3$, we have $n +2 \leq p^{\frac{n-1}{2}}$  and so, imitating the arguments in \eqref{eq_Gbarob}, we get that $\gamma_3(\overline{G}) = \overline{G}^p \gamma_n(\overline{G}) = \overline{G}^p$. In particular $\overline{G}$ is a $p$-obelisk satisfying $p = |\gamma_{n-1}(\overline{G}):\gamma_{n}(\overline{G})| = |\gamma_{n}(\overline{G}):\gamma_{n+1}(\overline{G})| = |\gamma_{n+1}(\overline{G}):\gamma_{n+2}(\overline{G})|$; contradiction to Proposition \ref{prop_ObProps}(1). The claim is now proven.

We conclude by observing that, when $n$ is such that $\gamma_n(G) = 1$, then $(\mathrm{i})$ yields $\gamma_3(G)=G^p$ and, since such an $n$ always exists, we are done.
\qed
\vspace{8pt}\\
\noindent
Within the family of $p$-obelisks, there is a natural separation that is dictated by the collection of sizes of minimal generating sets of maximal subgroups. 

\begin{definition}
Let $\cor{O}$ be a $p$-obelisk. Then $\cor{O}$ is \emph{framed} if all maximal subgroups of $\cor{O}$ are $2$-generated.
\end{definition}

\noindent
Though not referred to with the name ``framed $p$-obelisks'', these groups are the reason Blackburn became interested in $p$-obelisks in the first place \cite[Theorem 4.2]{Bla61}. However, not all obelisks are framed \cite[Lemma 437]{Sta17}
%\cite[Lemma 13.26]{Sta17} 
and in non-framed obelisk there are exactly $2$ maximal subgroups that are not $2$-generated (but $3$-generated) \cite[Lemma 335]{Sta17}. The following result \cite[Proposition 336]{Sta17}
%\cite[Proposition 10.14]{Sta17} 
makes it clear how the restricted $k$-Lie algebra of an obelisk determines whether the obelisk is framed or not.

\begin{lemma}\label{lemma:framed}
Let $\cor{O}$ be a $p$-obelisk. The following are equivalent:
\begin{enumerate}[label=$(\arabic*)$]
\item the obelisk $\cor{O}$ is framed, 
\item for each maximal subgroup $M$ of $\cor{O}$, the images of $M^p$ and $[M,M]$ in $\gamma_3(\cor{O})/\gamma_4(\cor{O})$ are distinct subgroups of order $p$. 
\end{enumerate}
\end{lemma}

\begin{theorem}\label{theo_FNF}
Let $\mathcal{O}$ be a framed $p$-obelisk such that $kG \cong k\mathcal{O}$. Then $G$ is a framed $p$-obelisk.
\end{theorem}
\begin{proof}
Thanks to Theorem \ref{theo_IsOb}, the group $G$ is an obelisk if and only if $\mathcal{O}$ is. Now from Lemma~\ref{lemma_DimSubOb} we know that $\gamma_3(G)/\gamma_4(G) = D_p(G)/D_{p+1}(G)$ and $\gamma_3(\cor{O})/\gamma_4(\cor{O}) = D_p(\cor{O})/D_{p+1}(\cor{O})$.
The combination of  Lemma \ref{lemma:framed} with Section \ref{sec:invariants}(7) yields that $G$ is framed if and only if $\mathcal{O}$ is. 
\end{proof}

\bibliographystyle{amsalpha}
\bibliography{MIP}

\providecommand{\bysame}{\leavevmode\hbox to3em{\hrulefill}\thinspace}
\providecommand{\MR}{\relax\ifhmode\unskip\space\fi MR }
% \MRhref is called by the amsart/book/proc definition of \MR.
\providecommand{\MRhref}[2]{%
  \href{http://www.ams.org/mathscinet-getitem?mr=#1}{#2}
}
\providecommand{\href}[2]{#2}
\begin{thebibliography}{DdSMS99}

\bibitem[Bag99]{Bag99}
C.~Bagi\'{n}ski, \emph{On the isomorphism problem for modular group algebras of
  elementary abelian-by-cyclic {$p$}-groups}, Colloq. Math. \textbf{82} (1999),
  no.~1, 125--136.

\bibitem[BC88]{BC88}
C.~Bagi\'{n}ski and A.~Caranti, \emph{The modular group algebras of
  {$p$}-groups of maximal class}, Canad. J. Math. \textbf{40} (1988), no.~6,
  1422--1435.

\bibitem[BdR20]{BdR20}
O.~Broche and \'{A}. del R\'{\i}o, \emph{The modular isomorphism problem for
  two generated groups of class two}, arxiv.org/abs/arXiv:2003.13281, \url{
  https://arxiv.org/abs/2003.13281}.

\bibitem[BEO19]{SmallGroupLibrary}
H.~U. Besche, B~Eick, and E.~O'Brien, \emph{{SmallGrp}: The {GAP Small Groups
  Library}, version 1.4.1}, https://gap-packages.github.io/smallgrp/, 2019.

\bibitem[BK07]{BK07}
C.~Bagi\'{n}ski and A.~Konovalov, \emph{The modular isomorphism problem for
  finite {$p$}-groups with a cyclic subgroup of index {$p^2$}}, Groups {S}t.
  {A}ndrews 2005. {V}ol. 1, London Math. Soc. Lecture Note Ser., vol. 339,
  Cambridge Univ. Press, Cambridge, 2007, pp.~186--193.

\bibitem[BK19]{BK19}
C.~Bagi\'{n}ski and J.~Kurdics, \emph{The modular group algebras of
  {$p$}-groups of maximal class {II}}, Comm. Algebra \textbf{47} (2019), no.~2,
  761--771.

\bibitem[BKRW99]{BKRW99}
F.~Bleher, W.~Kimmerle, K.~W. Roggenkamp, and M.~Wursthorn, \emph{Computational
  aspects of the isomorphism problem}, Algorithmic algebra and number theory
  ({H}eidelberg, 1997), Springer, Berlin, 1999, pp.~313--329.

\bibitem[Bla61]{Bla61}
N.~Blackburn, \emph{Generalizations of certain elementary theorems on
  {$p$}-groups}, Proc. London Math. Soc. (3) \textbf{11} (1961), 1--22.

\bibitem[Bra63]{Bra63}
R.~Brauer, \emph{Representations of finite groups}, Lectures on {M}odern
  {M}athematics, {V}ol. {I}, Wiley, New York, 1963, pp.~133--175.

\bibitem[DdSMS99]{DdSMS99}
J.~D. Dixon, M.~P.~F. du~Sautoy, A.~Mann, and D.~Segal, \emph{Analytic
  pro-{$p$} groups}, second ed., Cambridge Studies in Advanced Mathematics,
  vol.~61, Cambridge University Press, Cambridge, 1999.

\bibitem[Des56]{Des56}
W.~E. Deskins, \emph{Finite {A}belian groups with isomorphic group algebras},
  Duke Math. J. \textbf{23} (1956), 35--40.

\bibitem[Eic08]{Eic08}
B.~Eick, \emph{Computing automorphism groups and testing isomorphisms for
  modular group algebras}, J. Algebra \textbf{320} (2008), no.~11, 3895--3910.

\bibitem[EK11]{EK11}
B.~Eick and A.~Konovalov, \emph{The modular isomorphism problem for the groups
  of order 512}, Groups {S}t {A}ndrews 2009 in {B}ath. {V}olume 2, London Math.
  Soc. Lecture Note Ser., vol. 388, Cambridge Univ. Press, Cambridge, 2011,
  pp.~375--383.

\bibitem[GAP19]{GAP}
The GAP~Group, \emph{{GAP -- Groups, Algorithms, and Programming, Version
  4.10.2}}, 2019, http://www.gap-system.org.

\bibitem[HS06]{HS06}
M.~Hertweck and M.~Soriano, \emph{On the modular isomorphism problem: groups of
  order {$2^6$}}, Groups, rings and algebras, Contemp. Math., vol. 420, Amer.
  Math. Soc., Providence, RI, 2006, pp.~177--213.

\bibitem[HS07]{HS07}
\bysame, \emph{Parametrization of central {F}rattini extensions and
  isomorphisms of small group rings}, Israel J. Math. \textbf{157} (2007),
  63--102.

\bibitem[Hup67]{Hup67}
B.~Huppert, \emph{Endliche {G}ruppen. {I}}, Die Grundlehren der Mathematischen
  Wissenschaften, Band 134, Springer-Verlag, Berlin-New York, 1967.

\bibitem[Jam80]{Jam80}
R.~James, \emph{The groups of order {$p^{6}$} ({$p$} an odd prime)}, Math.
  Comp. \textbf{34} (1980), no.~150, 613--637.

\bibitem[Jen41]{Jen41}
S.~A. Jennings, \emph{The structure of the group ring of a {$p$}-group over a
  modular field}, Trans. Amer. Math. Soc. \textbf{50} (1941), 175--185.

\bibitem[JNO90]{JNOB90}
R.~James, M.~F. Newman, and E.~A. O'Brien, \emph{The groups of order {$128$}},
  J. Algebra \textbf{129} (1990), no.~1, 136--158.

\bibitem[MM20]{MM20}
L.~{Margolis} and T.~{Moede}, \emph{{The Modular Isomorphism Problem for small
  groups -- revisiting Eick's algorithm}}, arXiv:2010.07030,
  \url{https://arxiv.org/abs/2010.07030}.

\bibitem[OVL05]{OVL05}
E.~A. O'Brien and M.~R. Vaughan-Lee, \emph{The groups with order {$p^7$} for
  odd prime {$p$}}, J. Algebra \textbf{292} (2005), no.~1, 243--258.

\bibitem[Pas77]{Pas77}
D.~S. Passman, \emph{The algebraic structure of group rings}, Pure and Applied
  Mathematics, Wiley-Interscience [John Wiley \& Sons], New York-London-Sydney,
  1977.

\bibitem[PS72]{PS72}
I.~B.~S. Passi and S.~K. Sehgal, \emph{Isomorphism of modular group algebras},
  Math. Z. \textbf{129} (1972), 65--73.

\bibitem[Sak20]{Sak20}
T.~Sakurai, \emph{The isomorphism problem for group algebras: a criterion}, J.
  Group Theory \textbf{23} (2020), no.~3, 435--445.

\bibitem[Sal93]{Sal93}
M.~A.~M. Salim, \emph{The isomorphism problem for the modular group algebras of
  groups of order $p^5$}, Ph.D. thesis, Department of Mathematics, University
  of Manchester, 1993.

\bibitem[San85]{San84}
R.~Sandling, \emph{The isomorphism problem for group rings: a survey}, Orders
  and their applications ({O}berwolfach, 1984), Lecture Notes in Math., vol.
  1142, Springer, Berlin, 1985, pp.~256--288.

\bibitem[San89]{San89}
\bysame, \emph{The modular group algebra of a central-elementary-by-abelian
  {$p$}-group}, Arch. Math. (Basel) \textbf{52} (1989), no.~1, 22--27.

\bibitem[SS95]{SS95}
M.~A.~M. Salim and R.~Sandling, \emph{The unit group of the modular small group
  algebra}, Math. J. Okayama Univ. \textbf{37} (1995), 15--25.

\bibitem[SS96a]{SS96p5}
\bysame, \emph{The modular group algebra problem for groups of order {$p^5$}},
  J. Austral. Math. Soc. Ser. A \textbf{61} (1996), no.~2, 229--237.

\bibitem[SS96b]{SS96MC}
\bysame, \emph{The modular group algebra problem for small {$p$}-groups of
  maximal class}, Canad. J. Math. \textbf{48} (1996), no.~5, 1064--1078.

\bibitem[{Sta}17]{Sta17}
M.~{Stanojkovski}, \emph{{Intense automorphisms of finite groups}},
  arXiv:1710.08979, \url{https://arxiv.org/abs/1710.08979}, to appear in the
  Memoirs of the AMS.

\bibitem[Suz86]{Suz86}
M.~Suzuki, \emph{Group theory {II}}, Grundlehren der Mathematischen
  Wissenschaften [Fundamental Principles of Mathematical Sciences], vol. 248,
  Springer-Verlag, New York, 1986, Translated from the Japanese.

\bibitem[Wur93]{Wur93}
M.~Wursthorn, \emph{Isomorphisms of modular group algebras: an algorithm and
  its application to groups of order {$2^6$}}, J. Symbolic Comput. \textbf{15}
  (1993), no.~2, 211--227.

\end{thebibliography}

\vspace*{2em}
\noindent
{\footnotesize
\begin{minipage}[t]{0.5\textwidth}
  Leo Margolis \\
Vrije Universiteit Brussel   \\
Department of Mathematics\\
Pleinlaan 2 \\
1050 Brussels \\
Belgium\\
  \quad\\
  E-mail:  \href{mailto:}{leo.margolis@vub.be}
\end{minipage}
\hfill
\begin{minipage}[t]{0.5\textwidth}
  Mima Stanojkovski\\
  Max-Planck-Institut f\"ur Mathematik in \\
  den Naturwissenschaften\\
  Inselstrasse 22\\
  04103 Leipzig\\
  Germany\\
  \quad\\
  E-mail: \href{mailto:mima.stanojkovski@mis.mpg.de}{mima.stanojkovski@mis.mpg.de}
\end{minipage}
}

\end{document}